\documentclass[11pt]{article}

\usepackage[utf8]{inputenc}
\usepackage[british]{babel}
\usepackage[a4paper, margin=2.5cm]{geometry}
\usepackage{amsmath, amsthm, amssymb, amsfonts}
\usepackage{mathtools}
\usepackage{thmtools}
\usepackage[shortlabels, inline]{enumitem}
\usepackage[font={small, it}]{caption}
\usepackage{subcaption}
\usepackage{microtype}
\usepackage{xcolor}
\usepackage{mathabx}
\usepackage{hyperref}
\usepackage{bm}
\usepackage{bbm}

\newcommand{\cF}{\ensuremath{\mathcal F}}
\newcommand{\cG}{\ensuremath{\mathcal G}}
\newcommand{\cH}{\ensuremath{\mathcal H}}
\newcommand{\cI}{\ensuremath{\mathcal I}}

\newcommand{\cP}{\ensuremath{\mathcal P}}

\newcommand{\cR}{\ensuremath{\mathcal R}}
\newcommand{\cS}{\ensuremath{\mathcal S}}

\newcommand{\cU}{\ensuremath{\mathcal U}}

\newcommand{\cZ}{\ensuremath{\mathcal Z}}

\newcommand{\eps}{\varepsilon}
\renewcommand{\phi}{\varphi}
\renewcommand{\rho}{\varrho}

\DeclareMathOperator*{\N}{\mathbb{N}}

\DeclarePairedDelimiter\ceil{\lceil}{\rceil}
\DeclarePairedDelimiter\floor{\lfloor}{\rfloor}

\let\setminus=\smallsetminus

\newcommand{\Gnp}{G(n, p)}

\newcommand{\Hnp}{\mathcal{H}^3(n,p)}

\newcommand{\osref}[2]{%
  \setlength\abovedisplayskip{5pt plus 2pt minus 2pt}
  \setlength\abovedisplayshortskip{5pt plus 2pt minus 2pt}
  \ensuremath{\overset{\text{#1}}{#2}}
}

\declaretheorem[parent=section]{theorem}
\declaretheorem[sibling=theorem]{lemma}

\declaretheorem[sibling=theorem]{claim}

\declaretheorem[sibling=theorem,style=definition]{definition}

\setlist{itemsep=0.1em, topsep=0.1em, parsep=0.1em, partopsep=0.1em}

\hypersetup{
  colorlinks,
  linkcolor={red!60!black},
  citecolor={green!50!black},
  urlcolor={blue!80!black}
}

\colorlet{RoyalRed}{red!70!black}
\definecolor{RoyalBlue}{rgb}{0.25, 0.41, 0.88}
\definecolor{RoyalAzure}{rgb}{0.0, 0.22, 0.66}

\newlength{\bibitemsep}
\newlength{\bibparskip}
\let\oldthebibliography\thebibliography
\renewcommand\thebibliography[1]{%
  \oldthebibliography{#1}%
  \setlength{\parskip}{\bibitemsep}%
  \setlength{\itemsep}{\bibparskip}%
}

\newcounter{propcnt} 

\newlist{alphenum}{enumerate}{1}
\setlist[alphenum,1]{%
  label=\normalfont{\bfseries{(\Alph{propcnt}{\arabic*})}},
  ref=\normalfont{(\Alph{propcnt}{\arabic*})},
  leftmargin = \parindent+3.5em,
}

\usepackage{tikz}
\usetikzlibrary{topaths,calc}

\usepackage[ruled,vlined]{algorithm2e}
\SetKwInput{KwInput}{Input}                
\SetKwInput{KwOutput}{Output}              

\newcommand{\mbw}{\ensuremath{\mathbf{w}}}

\definecolor{ao(english)}{rgb}{0.0, 0.5, 0.0}

\title{Transference for loose Hamilton cycles in random $3$-uniform hypergraphs}
\author{
  Kalina Petrova\thanks{Institute of Theoretical Computer Science, ETH
    Z\"{u}rich, 8092 Z\"{u}rich, Switzerland \newline Email:
    \{\texttt{kpetrova}\textbar \texttt{mtrujic}\}\texttt{@inf.ethz.ch}} \textsuperscript{,}\thanks{Research supported by grant no.\ CRSII5 173721 of the Swiss National Science Foundation}
  \and
  Milo\v{s} Truji\'{c}\footnotemark[1] \textsuperscript{,}\thanks{Research
    supported by grant no.\ 200020 197138 of the Swiss National Science
  Foundation.}
}
\date{}

\begin{document}
\maketitle

\begin{abstract}
  A loose Hamilton cycle in a hypergraph is a cyclic sequence of edges covering
  all vertices in which only every two consecutive edges intersect and do so in
  exactly one vertex. With Dirac's theorem in mind, it is natural to ask what
  minimum $d$-degree condition guarantees the existence of a loose Hamilton
  cycle in a $k$-uniform hypergraph. For $k=3$ and each $d \in \{1,2\}$, the
  necessary and sufficient such condition is known precisely. We show that these
  results adhere to a `transference principle' to their sparse random analogues.
  The proof combines several ideas from the graph setting and relies on the
  absorbing method. In particular, we employ a novel approach of Kwan and Ferber
  for finding absorbers in subgraphs of sparse hypergraphs via a contraction
  procedure. In the case of $d = 2$, our findings are asymptotically optimal.
\end{abstract}

\section{Introduction}

The question of deciding when a given graph is Hamiltonian is in general
notoriously difficult and was included in Karp's original list of 21 NP-complete
problems~\cite{karp1972reducibility}. Being a fundamental problem in graph
theory (and computer science), Hamiltonicity has inspired a long line of work
exploring sufficient conditions for it. Perhaps the best known among those is
the classical theorem of Dirac~\cite{dirac1952some}: every graph on $n \geq 3$
vertices with minimum degree at least $n/2$ contains a Hamilton cycle. Another
regime in which the problem is understood better than in the general case is
that of random graphs (also, more broadly, quasi-random graphs and expanders).
In that regard, P{\'o}sa~\cite{posa1976hamiltonian} and independently
Korshunov~\cite{korshunov1976solution} proved that if $p \ge C\log n/n$, for
some constant $C > 0$, then the Erd\H{o}s-R\'{e}nyi binomial random graph $G(n,
p)$\footnote{$\Gnp$ stands for a graph on $n$ vertices in which each edge exists
with probability $p = p(n) \in (0,1)$ independently.} is \emph{with high
probability}\footnote{With high probability (or w.h.p.\ for brevity) means with
probability going to $1$ as $n$ tends to infinity.} Hamiltonian. The more
precise value $p = p(n)$ for which the former holds was later determined by
Koml{\'o}s and Szemer{\'e}di~\cite{komlos1983limit}, and an even stronger,
so-called \emph{hitting-time} result, was shown by Ajtai, Koml{\'o}s, and
Szemer{\'e}di~\cite{ajtai1985first} and independently by
Bollob{\'a}s~\cite{bollobas1984evolution}.

Inquiring further into properties of random graphs, Sudakov and
Vu~\cite{sudakov2008local} asked how \emph{resilient} $\Gnp$ is with respect to
having a Hamilton cycle. A graph $G$ is $\alpha$-resilient with respect to a
property $\cP$ if after the removal of at most an $\alpha$-fraction of edges
incident to every vertex of $G$, the resulting graph (still) contains $\cP$.
Observe that Dirac's theorem states exactly that: the complete graph on $n$
vertices $K_n$ is $1/2$-resilient with respect to Hamiltonicity (with $1/2$
being optimal as witnessed by two disjoint cliques of size $n/2$). The work of
Sudakov and Vu initiated a systematic study of minimum degree requirements in
the flavour of Dirac's theorem for random graphs as they showed that for $p \gg
\log^4 n/n$, w.h.p.\ $\Gnp$ is $(1/2-o(1))$-resilient with respect to
Hamiltonicity. A full analogue of Dirac's theorem for random graphs was later
established by Lee and Sudakov~\cite{lee2012dirac}: if $p \gg \log n/n$ then
w.h.p.\ $\Gnp$ is $(1/2-o(1))$-resilient with respect to Hamiltonicity. Even
more, the absolutely best-possible hitting-time results were recently obtained
by Montgomery~\cite{montgomery2019hamiltonicity} and independently Nenadov,
Steger, and the second author~\cite{nenadov2019resilience}.

There are, however, certain deficiencies of the basic notion of resilience in
the usual binomial random graph---for example, it is unable to capture the
behaviour of $\Gnp$ with respect to containment of large structures riddled with
triangles. Namely, as soon as $p = o(1)$, one can easily prevent as many as
$\Theta(p^{-2})$ vertices from being in a triangle while removing only $o(np)$
edges incident to each vertex (see~\cite{balogh2012corradi,
huang2012bandwidth}). This makes the study of resilience for, e.g., a
triangle-factor or any power of a Hamilton cycle futile. Recently, Fischer,
Steger, \v{S}kori\'{c}, and the second author extended the notion of resilience
to $H$-resilience in order to study robustness of $\Gnp$ with respect to the
containment of the square of a Hamilton cycle~\cite{fischer2022triangle}.
Roughly speaking, they determined the smallest $\alpha \in [0, 1]$ such that,
for $p \gg \log^3 n/\sqrt{n}$, w.h.p.\ every $G \subseteq \Gnp$ in which every
vertex belongs to at least $(\alpha + o(1))p^3 \binom{n}{2}$ triangles contains
the square of a Hamilton cycle.

Such a concept naturally corresponds to resilience in hypergraphs, with the
added advantage that edges in a random hypergraph are independent, unlike copies
of $H$ in $\Gnp$. Thus, the hypergraph counterpart to the questions above can be
a solid test case for developing new ideas and techniques. Of course, studying
resilience of Hamiltonicity in hypergraphs is also of independent interest as a
natural `high-dimensional' generalisation of one of the most important questions
in graph theory to a radically more challenging setting.

In this paper, we are concerned with Dirac-type conditions, hence resilience,
for Hamiltonicity in random $k$-uniform\footnote{A hypergraph is $k$-uniform,
also called $k$-\emph{graph} for short, if every edge consists of exactly $k$
vertices.} hypergraphs. The notions of cycles and degrees do not generalise
unambiguously to the hypergraph setting, so in what follows we make them more
specific. For some $1 \leq \ell < k$, an $\ell$-cycle in a $k$-graph is a cyclic
sequence of vertices such that every vertex belongs to some edge, every edge
consists of $k$ consecutive vertices, and each two consecutive edges overlap in 
exactly $\ell$ vertices. The most
studied special cases are $\ell=1$ and $\ell=k-1$, which are referred to as a
\emph{loose} and a \emph{tight} cycle, respectively. Note that the number of
vertices in any $\ell$-cycle is necessarily divisible\footnote{From now on, we
always assume that the hypergraph we are dealing with has its number of
vertices divisible by the appropriate integer needed for a loose Hamilton
cycle to exist.} by $k-\ell$. In a $k$-graph $\cH$, for some $1 \leq d < k$, the
$d$-degree $\deg_\cH(S)$ of a $d$-element set of vertices $S \subseteq V(\cH)$
is the number of edges $A \in E(\cH)$ such that $S \subseteq A$. The $1$-degree
$\deg_\cH(v)$ of a vertex $v$ is referred to simply as its degree, and the
$(k-1)$-degree of a $(k-1)$-set as its codegree. The minimum $d$-degree
$\delta_d(\cH)$ of $\cH$ is the minimum value among the $d$-degrees over all
$d$-sets $S \subseteq V(\cH)$.

As is the case with graphs, there is a large body of research studying various
Dirac-type conditions in hypergraphs, that is smallest $d$-degree which implies
existence of Hamilton $\ell$-cycles (e.g.~\cite{reiher2019minimum,
 rodl2006dirac, kuhn2010hamilton}, surveys~\cite{kuhn2014hamilton,
zhao2016recent} and the references within). Most closely related to the present
work, for $d = k - 1$, Keevash, K\"{u}hn, Mycroft, and
Osthus~\cite{keevash2011loose} and independently H\'{a}n and
Schacht~\cite{han2010dirac} showed that every (sufficiently large) $n$-vertex
$k$-graph $\cH$ with $\delta_{k-1}(\cH) \geq (\frac{1}{2(k-1)} + o(1))n$
contains a loose Hamilton cycle, which is asymptotically optimal (up to the
$o(n)$ term). In case of $3$-graphs, it was previously known that $\delta_d(\cH)
\geq (\delta_d + o(1))\binom{n-d}{3-d}$ implies (loose) Hamiltonicity, with
$\delta_1 = 7/16$ for $d = 1$ or $\delta_2 = 1/4$ for $d = 2$
(see~\cite{buss2013minimum, kuhn2006loose} and their strengthening
\cite{han2015minimum, czygrinow2014tight}). Neither of the values $\delta_1 =
7/16, \delta_2 = 1/4$ can be improved upon.

When it comes to (sparse) random structures, Dudek, Frieze, Loh, and
Speiss~\cite{dudek2012optimal} showed that
$\mathcal{H}^k(n,p)$\footnote{$\mathcal{H}^k(n,p)$ stands for the (random) graph
on $n$ vertices in which each set of $k$ vertices is chosen as an edge
independently with probability $p$.} w.h.p.\ contains a loose Hamilton cycle,
provided that $p \gg \log n/n^{k-1}$ (this generalises a previous result of
Frieze~\cite{frieze2010loose}; see also \cite{ferber2015closing} for a short
proof and \cite{park2023proof} for a much more general framework from which this
follows directly). This result is asymptotically optimal, since for $p$ of lower
order of magnitude, $\mathcal{H}^k(n,p)$ w.h.p.\ has isolated vertices.

Our main contribution is a resilience variant of this result for $3$-graphs, thus
transferring the formerly mentioned Dirac-type statements to the sparse random
setting.

\begin{theorem}\label{thm:main-theorem}
  Let $d \in \{1,2\}$. For every $\gamma > 0$ there is a $C > 0$ such that the
  following holds. Suppose $p \geq C\log n \cdot \max\{n^{-3/2}, n^{-3+d}\}$ and
  $n$ is even. Then $\Hnp$ w.h.p.\ has the property that every spanning subgraph
  $G \subseteq \Hnp$ with $\delta_d(G) \geq (\delta_d+\gamma)p\binom{n-d}{3-d}$
  contains a loose Hamilton cycle.
\end{theorem}

The constants $\delta_1 = 7/16$ and $\delta_2 = 1/4$ are optimal---in other
words, w.h.p.\ there exists a subgraph $G$ of $\Hnp$ with minimum $d$-degree
$(\delta_d - o(1))\binom{n-d}{3-d}$ which does not have a loose Hamilton cycle.
One can see this by, e.g., considering analogues of dense extremal constructions
from \cite{buss2013minimum} (in case $d = 1$) and \cite{kuhn2006loose} (in case
$d = 2$).

The more interesting part are the limitations for the density $p$. Observe that
one can more concisely summarise the density requirements as $p \geq C\log
n/n^{2-d/2}$, however we chose to present them separately to underline their
distinct origins. In case $d = 2$, if $p = c\log n/n$ for some small $c > 0$,
then w.h.p.\ $\Hnp$ contains pairs of vertices with $2$-degree equal to zero,
and thus Theorem~\ref{thm:main-theorem} is optimal (up to the multiplicative
constant). However, the bound $p \geq Cn^{-3/2}\log n$ for $d = 1$ is
`artificial' but crucial for our proof technique. It is highly likely, but
probably quite challenging to prove, that the result should remain true all the
way down to $p = \Omega(\log n/n^2)$---the threshold for appearance of a loose
Hamilton cycle in $\Hnp$. It seems that one would need to explore completely
different techniques in order to establish this.

Note that this provides a resilience result for $\Hnp$ with respect to
containment of a loose Hamilton cycle---it shows that $\Hnp$ is
$(9/16-o(1))$-resilient for this property (similarly one can think of `codegree
resilience' in terms of $\delta_2$). This partially answers a question raised by
Frieze~\cite[Problem~56]{frieze2021hamilton} in his survey on Hamilton cycles in
random graphs. Being only the beginning of the story, a natural next step would
be an extension to hypergraphs of larger uniformities. One obstacle on the way
is that the corresponding dense Dirac condition with $d=1$ is not known for
$k$-graphs with $k \geq 4$ (a `transference' result could be obtained even
without knowing this value, see~\cite{ferber2022dirac}). In addition, our
methods do not seem to straightforwardly generalise to $k \geq 4$ regardless of
the type of degree considered.

Usually, establishing properties of hypergraphs is substantially more difficult
than in the case of their graph counterparts, hence it is not surprising that
not much is known about resilience of random hypergraphs. Prior to this,
Clemens, Ehrenm\"{u}ller, and Person~\cite{clemens2020dirac} studied resilience
of $\cH^k(n,p)$ with respect to the containment of a Hamilton Berge cycle, and, very
recently, Allen, Parczyk, and Pfenninger~\cite{allen2021resilience} proved that
$\cH^k(n,p)$ is resilient for containment of tight Hamilton cycles. On a related
note, Ferber and Kwan~\cite{ferber2022dirac} proved a very general `transference
principle' concerning perfect matchings in random $k$-graphs (see also
\cite{ferber2020co} for a specific result when $d = k - 1$).

On the whole, our proof follows a standard strategy for embedding relying on the
\emph{absorbing method}. In a nutshell, this method allows one to reduce the
problem of finding a \emph{spanning} structure to the usually significantly
easier one of finding an \emph{almost-spanning} one. For the latter we combine
several ideas originating from their graph counterparts such as the DFS
technique for finding long paths (Section~\ref{sec:path-cover}), path connection
techniques (Section~\ref{sec:connecting-lemma}), and the sparse regularity
method (Section~\ref{sec:sparse-regularity}). As always, the most difficult,
involved, and creative part comes in designing and finding the absorber (see
Definition~\ref{def:absorbing-graph} below). The `finding' part is partially
done through a contraction procedure of Ferber and Kwan which helps with finding
absorbers inside a regular partition of $\Hnp$. We discuss this, as well as the
absorbing method in general, in much greater detail in
Section~\ref{sec:absorbing-method}. All these ingredients are mixed together
following a usual recipe to give a proof of Theorem~\ref{thm:main-theorem}
(Section~\ref{sec:main-proof}).

\section{Preliminaries}

Our graph theoretic notation mostly follows standard textbooks in the area,
e.g.~\cite{bondy2008graph}. More specifically, for a (hyper)graph $G = (V, E)$
we let $v(G)$ and $e(G)$ denote the number of its vertices and edges,
respectively. Given a set of vertices $X \subseteq V(G)$, $G - X$ stands for the
induced graph $G[V(G) \setminus X]$. For sets $X, Y, Z \subseteq V(G)$ we let
$e_G(X, Y, Z)$ denote the number of triples $(x, y, z) \in X \times Y \times Z$
for which $xyz \in E(G)$. Instead of $e_G(\{x\}, Y, Z)$, we write $e_G(x, Y, Z)$
for brevity. Similarly, if $\cP \subseteq \binom{V(G)}{2}$ is a set of
\emph{pairs} of vertices, then $e_G(\cP, Z)$ counts the number of (ordered)
pairs $(\{x,y\}, z)$ for which $xyz \in E(G)$ with $\{x,y\} \in \cP$ and $z \in
Z$. For a set $W \subseteq V(G)$, we use $\deg_G(x, W)$, respectively
$\deg_G(\{x,y\},W)$, for the number of distinct pairs $\{y, z\}$ with $y, z \in
W$, respectively the number of vertices $z \in W$, for which $xyz \in E(G)$.
The neighbourhood $N_G(v, W)$ of a vertex $v$ in a set $W \subseteq V(G)$ refers
to the set of vertices $w \in W$ such that $uvw\in G$ for some $u \in W$, and we
denote $N_G(v, V(G))$ as $N_G(v)$ for brevity. Similarly, the neighbourhood
$N_G(\{v,u\},W)$ of a pair of vertices $v,u$ in a set $W \subseteq V(G)$ denotes
the set of vertices $w \in W$ with $uvw \in G$, and we abbreviate
$N_G(\{v,u\},V(G))$ to $N_G(v,u)$.

A \emph{loose path} of length $\ell \in \N$ is an ordered sequence of distinct
vertices $v_1,\dotsc,v_{2\ell+1}$ and $\ell$ edges: $v_{2i-1}v_{2i}v_{2i+1} \in
E(G)$ for all $i \in [\ell]$. Note that a loose path always consists of an odd
number of vertices. Throughout, whenever we say path we mean loose path and we
write $xy$-path for a path whose start- and end-points are $x$ and $y$. A
(loose) cycle is defined similarly.

All logarithms are in base $e$. For $n \in \N$, $[n]$ stands for the set of
first $n$ integers, that is $[n] := \{1,\dotsc,n\}$. We use standard asymptotic
notation $o, O, \omega, \Omega$, and $\Theta$. For a set $W$ and an integer $d$,
$\binom{W}{d}$ denotes the collection of all $d$-element subsets, $d$-sets for
brevity, of $W$, and $W^d$ denotes the collection of all distinct $d$-tuples
$(w_1,\dotsc,w_d)$ with $w_i \in W$ and $w_i \neq w_j$ for each $1 \leq i < j
\leq d$. When using set-theoretic notation, we treat tuples as corresponding
sets, e.g.\ $x \in (w_1,\dotsc,w_d)$ stands for $x \in \{w_1,\dotsc,w_d\}$ and
two tuples are disjoint if they do not share an element.

\subsection{Distribution of edges}

In this subsection we list some lemmas that give upper bounds on the number of
edges between various vertex sets in $\Hnp$.

\begin{lemma}\label{lem:varying-size-sets}
  For every $\gamma > 0$ and $C > 0$ there exists $\lambda > 0$ such that
  w.h.p.\ $G \sim \Hnp$ satisfies the following, provided that $n^2p \gg \log
  n$. There are no two disjoint sets $A, B \subseteq V(G)$ of sizes $a$ and $b$
  such that $1 \leq a \leq \lambda n$, $b \leq Ca$, and $e_G(A, B, V(G)) \geq
  \gamma a p \binom{n-1}{2}$.
\end{lemma}
\begin{proof}
  Set $m := \gamma a p \binom{n-1}{2}$. Fix $a \leq \lambda n$, $b \leq Ca$, and
  two disjoint sets $A$ and $B$ of sizes $a$ and $b$, respectively. The
  probability that $e_G(A, B, V(G)) \geq m$ is at most
  \[
    \binom{a b (n-2)}{m} p^m \leq \Big(\frac{2e a b (n-2)}{\gamma a p
    (n-1)(n-2)}\Big)^m p^m \leq \Big(\frac{2eb}{\gamma (n-1)}\Big)^m \leq
    \Big(\frac{4e\lambda C}{\gamma}\Big)^m.
  \]
  This is at most $e^{-\omega(a \log n)}$ provided $\lambda$ is chosen small
  enough with respect to $\gamma$ and $C$. Then the union bound over at most
  $e^{(a+b)\log{n}} \leq e^{2Ca\log{n}}$ choices for the sets $A$ and $B$
  completes the proof.
\end{proof}

\begin{lemma}\label{lem:1-edge-concentration}
  For every $\eps \in (0, 1/300)$, w.h.p.\ $G \sim \Hnp$ satisfies the
  following. Let $X, Y \subseteq V(G)$ be sets of size $x$ and $y$.
  \begin{enumerate}[label=(\textit{\roman*})]
    \item\label{lem:1-edge-concentration-p1} If $y \leq x \leq \eps^{-3}\log
      n/(np)$ then $e_G(X, Y, V(G)) \leq \eps^{-4}x \log n$.
    \item\label{lem:1-edge-concentration-p2} If $x, y \geq \eps^{-3}\log
      n/(np)$ then $e_G(X, Y, V(G)) \leq (1+\eps)xynp$.
  \end{enumerate}
\end{lemma}
\begin{proof}
  Let $t := \eps^{-3}\log n/(np)$. The probability that there exist $X$, $Y$ of
  sizes $y \leq x \leq t$ for which $e_G(X, Y, V(G)) > \eps^{-4}x \log n =: m$
  is at most
  \begin{align*}
    \sum_{x=1}^{t} \binom{n}{x} \sum_{y=1}^{x} \binom{n}{y} \binom{xyn}{m}
    p^m \leq \sum_{x=1}^{t} \sum_{y=1}^{x} n^{2x} \Big(\frac{ex^2np}{m}\Big)^m
    & \leq \sum_{x=1}^{t} \sum_{y=1}^{x} n^{2x} \Big(\frac{ex t np}{\eps^{-4} x
    \log{n}}\Big)^m \\
    & \leq \sum_{x=1}^{t} \sum_{y=1}^{x} n^{2x} e^{-10x\log n} = o(1).
  \end{align*}
  If $x, y \geq t$, from Chernoff's inequality (see,
  e.g.~\cite[Corollary~2.3]{janson2011random})
  \[
    \Pr[e_G(X, Y, V(G)) > (1+\eps)xynp] \leq e^{-\frac{\eps^2}{3}xynp}.
  \]
  Then by the union bound, the probability that the property of the lemma fails
  is at most
  \begin{align*}
    \sum_{x=t}^{n} \binom{n}{x} \sum_{y=t}^{n} \binom{n}{y} \cdot
    e^{-\frac{\eps^2}{3} xynp}
    & \leq \sum_{x=t}^{n} \sum_{y=t}^{n} e^{x\log n} e^{y\log n} \cdot
    e^{-\frac{\eps^2}{3}xynp} \\
    & \leq \sum_{x=t}^{n} \sum_{y=t}^{n} e^{2\max\{x,y\}\log n} \cdot
    e^{-\frac{1}{3\eps}\max\{x,y\} \log n} = o(1),
  \end{align*}
  once again.
\end{proof}

Proving the next two lemmas follows similar steps as for the one above, using a
straightforward application of Chernoff's inequality and the union bound. Thus,
we omit the proofs.

\begin{lemma}\label{lem:2-edge-concentration}
  For every $\eps \in (0, 1/300)$, w.h.p.\ $G \sim \Hnp$ satisfies the following.
  Let $W \subseteq V(G)$ and $\cP \subseteq \binom{V(G) \setminus W}{2}$.
  \begin{enumerate}[label=(\textit{\roman*})]
    \item\label{lem:2-edge-concentration-p1} If $|W| \leq |\cP| \leq
      \eps^{-3}\log n/p$ then $e_G(\cP, W) \leq \eps^{-4}|\cP|\log n$.
    \item\label{lem:2-edge-concentration-p2} If $|W|, |\cP| \geq \eps^{-3}\log
      n/p$ then $e_G(\cP, W) \leq (1+\eps)|\cP||W|p$.
  \end{enumerate}
\end{lemma}

\begin{lemma}\label{lem:general-edge-concentration}
  For every $\eps > 0$, w.h.p.\ $G \sim \Hnp$ satisfies the following. Let $X, Y,
  Z \subseteq V(G)$ be (not necessarily disjoint) sets of size $x$, $y$, $z$,
  such that $xyzp \geq 200\eps^{-2}n$. Then
  \[
    e_G(X, Y, Z) \leq (1+\eps)xyzp.
  \]
  In particular, if $Y = Z$, then $e_G(X, \binom{Y}{2}) \leq
  (1+\eps)x\binom{y}{2}p$.
\end{lemma}

\subsection{Expansion}

The following lemma captures the fact that expansion of vertices behaves as
expected in a not too sparse subgraph $G$ of the random graph $\Hnp$. Namely, if
a vertex $v$ has the property that for some $W$ the minimum degree of $G[\{v\}
\cup W]$ is at least $\Omega(n^2p)$, then there are at least $\sqrt{n}$ vertices
$w_1 \in W$ and $\Theta(n)$ vertices $w_2 \in W$ for which there is a
$vw_i$-path of length $i$ in $G[\{v\} \cup W]$. It plays an important role both
for proving the Connecting Lemma and finding absorbers.

\begin{lemma}\label{lem:large-second-neighbourhood}
  For every $\gamma > 0$ there exist $\xi, C > 0$ such that w.h.p.\ $\Hnp$
  satisfies the following, provided that $p \geq Cn^{-3/2}\log n$. Let $G
  \subseteq \Hnp$ and $W \subseteq V(G)$ with $\delta_1(G[W]) \geq \gamma
  p \binom{n-1}{2}$. For every $x \in V(G)$ for which $\deg_G(x, W) \geq \gamma
  p \binom{n-1}{2}$, there exist $\cF_x \subseteq (W \setminus \{x\})^4$ and
  $\cP_x \subseteq \binom{W \setminus \{x\}}{2}$, all pairwise disjoint, of size
  $|\cF_x| \geq \xi n$, $|\cP_x| = \sqrt{n}$, and such that
  \stepcounter{propcnt}
  \begin{alphenum}
    \item $xuv \in E(G)$, for every $\{u, v\} \in \cP_x$, and
    \item for every $(w_1,w_2,w_3,w_4) \in \cF_x$ there is some $\{u, v\} \in
      \cP_x$ for which $uw_1w_2, vw_3w_4 \in E(G)$.
  \end{alphenum}
\end{lemma}
\begin{proof}
  Suppose $\Hnp$ has the properties of Lemma~\ref{lem:1-edge-concentration} and
  Lemma~\ref{lem:general-edge-concentration} for, say, $\eps = 10^{-3}$. Let
  $\cP_x \subseteq \binom{W \setminus \{x\}}{2}$ be a maximal set of disjoint
  pairs which close an edge together with $x$, and suppose $|\cP_x| < \sqrt{n}$.
  Let $X$ be a superset of all the vertices that belong to these pairs of size
  precisely $2\sqrt{n}$. Then, in particular, there are no edges $xuv \in E(G)$
  with $u, v \in W \setminus X$. Note that $2\sqrt{n} \geq \eps^{-3}\log n/(np)$
  by our choice of $p$ for $C > 0$ large enough. From the minimum degree
  assumption on the one hand, and the property of
  Lemma~\ref{lem:1-edge-concentration}~\ref{lem:1-edge-concentration-p2}
  (considering a superset of $\{x\}$ of size $\eps^{-3} \log{n}/(np)$ if
  necessary) on the other, we have
  \[
    \gamma p\binom{n-1}{2} \leq e_G(x, X, V(G)) \leq (1+\eps) \cdot
    \max\{\eps^{-3}\log n/(np), 1\} \cdot 2\sqrt{n} \cdot np.
  \]
  This is a contradiction for $C$ large enough.

  Similarly, after fixing $\cP_x$ and $X$, let $\cF_x$ be a maximal set of
  disjoint $4$-tuples which satisfy the second property of the lemma, and
  suppose $|\cF_x| < \xi n$. Let $Y$ be the vertices that belong to these
  $4$-tuples. We first show there is a set $\cP_x' \subseteq \cP_x$ of size
  $|\cP_x'| > |\cP_x|/2$ such that for every $\{u,v\} \in \cP_x'$ there exist
  distinct (also from other such vertices) $w_1, w_2 \in W \setminus (X \cup Y
  \cup \{x\})$ which close an edge in $G$ with $u$. Let $\cP_x'$ be the largest
  such set and let $Q$ denote the union of $X \cup Y \cup \{x\}$ and all such
  $w_1, w_2$ (the ones that $u \in \cP_x'$ closes an edge with). Then, with $X'$
  denoting the vertices in $\cP_x'$, we have
  \[
    e_G(X \setminus X', Q, W) \geq |X \setminus X'|\gamma n^2p/3.
  \]
  On the other hand, by the property of
  Lemma~\ref{lem:1-edge-concentration}~\ref{lem:1-edge-concentration-p2},
  assuming $|X \setminus X'| \geq \sqrt{n} \geq \eps^{-3}\log n/(np)$ and as
  $|Q| = 6\xi n$ (taking a superset if necessary), we have
  \[
    e_G(X \setminus X', Q, W) \leq e_G(X \setminus X', Q, V(G)) \leq (1+\eps) |X
    \setminus X'| 6\xi n^2p,
  \]
  leading to a contradiction for $\xi$ small enough. We just need to show that
  there is $\{u,v\} \in \cP_x'$ and $w_3, w_4 \in W \setminus Q$ which
  comprise an edge in $G$ with $v$. Indeed, exactly the same computation as
  above establishes this, which completes the proof.
\end{proof}

\subsection{Degree inheritance properties}

We first state a slightly strengthened version of
\cite[Lemma~5.5]{ferber2022dirac}. The strengthening comes in the bound on $p$,
which is stated to match the one from Theorem~\ref{thm:main-theorem} (and in
fact, most of the statements in this paper). For the case $d = 2$ we require $p
= \Theta(\log n/n)$, in contrast to $\Omega(n^{-2}\log^3 n)$, and here the
desired property is in fact easily obtained through Chernoff's inequality and
the union bound due to independence. In case $d = 1$ the former requirement of
$p \geq n^{-2}\log^3 n$ is actually necessary for this particular proof.

\begin{lemma}\label{lem:random-set-degree-inh}
  For every $\gamma, \mu, \xi > 0$, there is a $C$ such that w.h.p.\ $\Hnp$
  satisfies the following, provided that $p \geq C\log n \cdot \max\{ n^{-3/2},
  n^{-3+d} \}$. Let $G \subseteq \Hnp$ and let $W \subseteq V(G)$ be a uniformly
  random set of size at least $\xi n$. Then w.h.p.\ every $d$-set $S \subseteq
  V(G)$ that satisfies $\deg_G(S) \geq (\mu+\gamma)p\binom{|V(G)|-d}{3-d}$ also
  satisfies $\deg_G(S, W) \geq (\mu+\gamma/2)p\binom{|W|-d}{3-d}$.
\end{lemma}
\begin{proof}
  For $d = 2$, the codegree of a pair of vertices into $W$ follows a
  hypergeometric distribution with mean $\Omega(np)$, for which Chernoff's
  inequality applies (see, e.g.~\cite[Theorem~2.10]{janson2011random}). The
  statement thus follows directly from it and the union bound over all pairs of
  vertices. In case $d = 1$ the assertion holds by
  \cite[Lemma~5.5]{ferber2022dirac}, since $p \geq Cn^{-3/2}\log n \geq
  n^{-2}\log^3 n$.
\end{proof}

The next lemma allows us to start with a graph in which almost all $d$-sets have
at least some degree and pick a random subgraph of it such that in it, with positive
probability, \emph{all} $d$-sets have at least some (slightly smaller) degree.

\begin{lemma}[{\cite[Lemma~3.4]{ferber2022dirac}}]
  \label{lem:random-set-almost-all-degree-inh}
  There is a $c > 0$ such that the following holds. Let $d \in \{1,2\}$ and
  $\gamma,\delta,\mu > 0$. Let $G$ be an $n$-vertex $3$-graph in which all but
  $\delta\binom{n}{d}$ of the $d$-sets have degree at least
  $(\mu+\gamma)\binom{n-d}{3-d}$. Let $S$ be a uniformly random subset of $s
  \geq 2d$ vertices of $G$. Then with probability at least $1 - \binom{s}{d}
  (\delta + e^{-c \gamma^2 s})$, the random induced subgraph $G[S]$ has minimum
  $d$-degree at least $(\mu+\gamma/2)\binom{s-d}{3-d}$.
\end{lemma}

The last two lemmas establish that in a subgraph of the random hypergraph $\Hnp$
with sufficiently large minimum degree, after an adversary removes $\lambda n$
vertices, for some tiny $\lambda > 0$, almost all $d$-sets still keep a
significant portion of their original degree in the resulting graph.

\begin{lemma}\label{lem:robust-degree-inh}
  For every $\gamma, \mu > 0$ there exist $\lambda, C > 0$ such that w.h.p.\
  $\Hnp$ satisfies the following, provided that $p \geq Cn^{-3/2}\log n$. Let $G
  \subseteq \Hnp$ with $\delta_1(G) \geq (\mu+\gamma)p\binom{n-1}{2}$ and let $S
  \subseteq V(G)$ be of size $|S| \leq \lambda n$. Then there exists a set $T$
  of size at most $\sqrt{n}$ such that for $U := V(G) \setminus (S \cup T)$, the
  graph $G[U]$ is of minimum degree at least $(\mu+\gamma/4)p\binom{|U|-1}{2}$.
  Furthermore, if for some $x, y \in V(G) \setminus S$ we have $\deg_G(x, V(G)
  \setminus S), \deg_G(y, V(G) \setminus S) \geq (\mu+\gamma/2)p\binom{n-1}{2}$,
  then $T$ can be chosen to avoid $x, y$.
\end{lemma}

\begin{proof}
  Initially, let $T := \varnothing$. As long as there exists $v \in V(G)
  \setminus (S \cup T)$ and $v \neq x, y$ with
  \[
    \deg_G(v, V(G) \setminus (S \cup T)) \leq (\mu+\gamma/2)p\binom{n-1}{2},
  \]
  add such a vertex to $T$. Stop this process at the first point in time when
  $|T| = \eps^{-3}\log n/(np)$, for some small enough $\eps > 0$. Note that by
  the assumption on $p$, it also holds that $|T| < \sqrt{n}$. We then have
  \[
    e_G(T, S \cup T, V(G)) \geq |T| \cdot (\gamma/2)p\binom{n-1}{2} \geq
    (\gamma/8)|T|n^2p.
  \]
  On the other hand, as $\Hnp$ w.h.p.\ satisfies the conclusion of
  Lemma~\ref{lem:1-edge-concentration}~\ref{lem:1-edge-concentration-p2},
  \[
    e_G(T, S \cup T, V(G)) \leq (1+\eps)|T||S \cup T|np \leq (1+\eps)2\lambda
    |T| n^2p,
  \]
  which is a contradiction with the former, for $\lambda > 0$ small enough.

  Consider now $x$ and its degree into $V(G) \setminus (S \cup T)$. If it does
  not satisfy the bound promised by the lemma, this means
  \[
    \deg_G(x, T, V(G) \setminus S) \geq (\gamma/4)p\binom{n-1}{2} \geq
    (\gamma/16)n^2p.
  \]
  On the other hand, by the property of
  Lemma~\ref{lem:1-edge-concentration}~\ref{lem:1-edge-concentration-p1},
  \[
    e_G(x, T, V(G) \setminus S) \leq e_G(x, T, V(G)) \leq \eps^{-4}|T|\log n
    \leq \eps^{-7}\log^2 n/(np).
  \]
  This is a contradiction with the former for $C > 0$ large enough as $n^3p^2
  \geq C^2\log^2 n$ by the assumption on $p$.
\end{proof}

\begin{lemma}\label{lem:robust-codegree-inh}
  For every $\gamma, \mu > 0$, there exist $\lambda, C > 0$ such that for every
  $p \geq C\log n/n$ w.h.p.\ $\Hnp$ satisfies the following. Let $G \subseteq
  \Hnp$ with $\delta_2(G) \geq (\mu+\gamma)p(n-2)$. Then, for every $S \subseteq
  V(G)$ with $|S| \leq \lambda n$, all but at most $10^9\log n/p$ pairs of
  vertices $u, v \in V(G) \setminus S$ have $\deg_{G-S}(u,v) \geq
  (\mu+\gamma/2)p(n-|S|-2)$.
\end{lemma}

\begin{proof}
  Let
  \[
    \cP := \big\{ \{u,v\} : u, v \in V(G) \setminus S, \deg_G(\{u,v\}, S) \geq
    \gamma np/4 \big\}.
  \]
  As $\Hnp$ w.h.p.\ has the property of
  Lemma~\ref{lem:2-edge-concentration}~\ref{lem:2-edge-concentration-p2} for
  $\eps = 10^{-3}$, assuming $|\cP| > 10^9\log n/p$ we have
  \[
    |\cP| \cdot \gamma np/4 \leq e_G(\cP, S) \leq (1+\eps)|\cP| \max\{|S|p,
    10^{9}\log n\}.
  \]
  This leads to a contradiction for $\lambda > 0$ sufficiently small and $C$
  large enough, as $|S| \leq \lambda n$ and $np \gg \log n$.
\end{proof}

\subsection{(Hyper)graph theory}

The following Dirac-type conditions for the existence of a loose Hamilton cycle
were mentioned in the introduction but we state them here explicitly in the form
in which we use them later. Recall, $\delta_1 = 7/16$ and $\delta_2 = 1/4$.

\begin{theorem}[\cite{buss2013minimum, kuhn2006loose}]
\label{thm:dense-graph}
  Let $d \in \{1,2\}$ and $\gamma > 0$. Every sufficiently large $3$-uniform
  hypergraph $\cH$ on an even number of vertices $n$ with $\delta_d(\cH) \geq
  (\delta_d + \gamma )\binom{n-d}{3-d}$ contains a loose Hamilton cycle.
\end{theorem}

The next lemma gives a minimum degree condition for a $3$-graph that ensures
that any pair of vertices is contained in some loose cycle of length three. We
make use of it later for finding absorbers.

\begin{lemma}
  \label{lem:switcher-dense-graph}
  Let $\cH$ be a sufficiently large $n$-vertex $3$-uniform hypergraph which
  satisfies $\delta_1(\cH) \geq \frac{7}{16} \binom{n}{2}$. Then for every pair
  of vertices $u,v \in V(\cH)$, there exist distinct vertices $a_1, a_2, a_3,
  a_4 \in V(\cH) \setminus \{u,v\}$ such that $a_1a_2a_3, a_2a_4u, a_3a_4v \in
  E(\cH)$, that is $a_3a_1a_2ua_4va_3$ is a loose cycle.
\end{lemma}
\begin{proof}
  Suppose $\cH'$ is a sufficiently large $n$-vertex $3$-graph with
  $\delta_1(\cH') \geq (5/8+\gamma)^2\binom{n}{2}$, for some $\gamma > 0$. A
  simple counting argument (see, e.g., \cite[Claim~9]{buss2013minimum}) shows
  that for every two $u, v \in V(\cH')$, there is a set $W = W(u, v) \subseteq
  V(\cH')$ of size $\gamma n$, such that either:
  \begin{itemize}
    \item $\deg_{\cH'}(u,w) \geq \gamma n$ and $\deg_{\cH'}(v,w) \geq 3n/8$ for all
      $w \in W$, or
    \item $\deg_{\cH'}(v,w) \geq \gamma n$ and $\deg_{\cH'}(u,w) \geq 3n/8$ for
      all $w \in W$.
  \end{itemize}
  Since $7/16 = (5/8+\gamma)^2$ for $\gamma = \sqrt{7}/4-5/8 > 0$, by the
  discussion above there is a $W \subseteq V(\cH)$ such that, without loss of
  generality, $\deg_\cH(u,w) \geq \gamma n$ and $\deg_\cH(v,w) \geq 3n/8$ for
  all $w \in W$. Pick an arbitrary $a_4 \in W$ and an arbitrary $a_2 \in V(\cH)
  \setminus \{v\}$ such that $ua_4a_2 \in E(\cH)$. Note that $\deg_\cH(v, a_4)
  \geq 3n/8$ and $|N_\cH(a_2)| \geq \sqrt{7}n/4$ due to the minimum degree
  condition for $a_2$. Therefore, since $\sqrt{7}/4 + 3/8 > 1$, we have
  $|N_\cH(v, a_4) \cap N_\cH(a_2)| > 2$ (with plenty of room to spare). Pick
  $a_3 \in N_\cH(v, a_4) \cap N_\cH(a_2) \setminus \{u\}$ and pick $a_1 \in
  N_\cH(a_2,a_3) \setminus \{a_4,u,v\}$.
\end{proof}

We lastly need a Hall-type matching condition for `bipartite' hypergraphs due to
Haxell, which has been used frequently for embedding problems in random graph
theory.

\begin{theorem}[Haxell's condition~\cite{haxell1995condition}]
  \label{haxell}
  Let $\cH$ be an $\ell$-graph whose vertex set can be partitioned into sets $A$
  and $B$ such that $|e \cap A| = 1$ and $|e \cap B| =
  \ell-1$, for every edge $e \in E(\cH)$. Suppose that for every choice of
  subsets $A' \subseteq A$ and $B' \subseteq B$ such that $|B'| \leq (2\ell-
  3)(|A'|-1)$, there exists an edge $e \in E(\cH)$ intersecting $A'$ but not
  $B'$. Then $\cH$ has an $A$-saturating matching (i.e.\ a collection of
  disjoint edges whose union contains $A$).
\end{theorem}

\section{The sparse regularity method for
hypergraphs}\label{sec:sparse-regularity}

Following \cite{ferber2022dirac, ferber2020almost} in a natural generalisation
of the analogous concept for graphs, we say that, given $\eps > 0$ and $p \in
(0,1)$, a $3$-partite $3$-graph $G$ on sets $V_1, V_2, V_3$ is
$(\eps,p)$-regular if for every $X_i \subseteq V_i$, $|X_i| \geq \eps|V_i|$, we
have
\[
  \big|d_G(V_1,V_2,V_3) - d_G(X_1,X_2,X_3)\big| \leq \eps p,
\]
where $d_G(A, B, C) = e_G(A, B, C)/(|A||B||C|)$ stands for the density of edges
of a given triple.

A partition $(V_i)_{i\in [t]}$ of the vertex set of a $3$-graph $G$ is said to
be $(\eps,p)$-\emph{regular} if it is an equipartition and for all but at most
$\eps\binom{t}{3}$ triples $V_i,V_j,V_k$, the graph induced by them is
$(\eps,p)$-regular. For the (hyper)graph regularity lemma to be of any use, it
usually needs to prevent too many edges lying within some partition class $V_i$.
A common way to restrict this is the notion of \emph{upper-uniformity}. We say
that a $3$-graph $G$ is $(\eta,b,p)$-upper-uniform, for some $\eta \in (0, 1)$
and $b \geq 1$, if $d_G(V_1,V_2,V_3) \leq bp$ for all disjoint sets
$V_1,V_2,V_3$ with $|V_i| \geq \eta|V(G)|$.

With all these concepts at hand, we state a so-called weak\footnote{`Weak' comes
from the fact that in this variant, the corresponding counting and embedding
lemmas are not necessarily true in general---one would require a stronger
concept of regularity.} hypergraph regularity lemma, which acts as a natural
generalisation from the graph setting, can be proven in the same way (see,
e.g.~\cite{kohayakawa1997szemeredi, kohayakawa2003szemeredi, gerke2005sparse}
for the sparse regularity lemma and \cite{kohayakawa2010weak} for the regularity
lemma in dense hypergraphs), and appears in the same form in
\cite{ferber2022dirac, ferber2020almost}.

\begin{theorem}\label{thm:sparse-hypergraph-regularity}
  For every $\eps > 0$ and $t_0, b \geq 1$ there exist $\eta > 0$ and $T \geq
  t_0$ such that for every $p \in (0, 1]$, every $(\eta,b,p)$-upper-uniform
  $3$-graph $G$ with at least $t_0$ vertices admits an $(\eps,p)$-regular
  partition $(V_i)_{i \in [t]}$, where $t_0 \leq t \leq T$.
\end{theorem}

The upper-uniformity property can be seen as a `true property of random graphs'
and indeed is exhibited by $\Hnp$ with high probability (e.g.\ established by a
straightforward application of Chernoff's inequality and the union bound).

\begin{lemma}\label{lem:hnp-is-upper-uniform}
  For every $\eta \in (0,1)$ and $b > 1$ the random $3$-graph $\Hnp$ is
  $(\eta,b,p)$-upper-uniform with probability at least $1-e^{-\omega(n\log n)}$,
  provided that $p = \omega(n^{-2}\log n)$.
\end{lemma}

Given an equipartition $(V_i)_{i \in [t]}$ of the vertex set of a $3$-graph $G$,
we define the reduced graph $\cR = \cR((V_i)_{i \in [t]}, \eps, p, \alpha)$ on
vertex set $\{1,\dotsc,t\}$ corresponding to the sets $V_i$, whose edges are
all $3$-element sets of indices $\{i,j,k\}$ such that $d_G(V_i,V_j,V_k) \geq
\alpha p$ and $G[V_i,V_j,V_k]$ is $(\eps,p)$-regular.

In fact, in the regularity lemma, one can even roughly define where the clusters
$V_i$ lie in the graph $G$. Namely, given a partition $V(G) = P_1 \cup \dotsb
\cup P_h$, the $(\eps,p)$-regular partition resulting from
Theorem~\ref{thm:sparse-hypergraph-regularity} can be made such that all but at
most $\eps h t$ clusters $V_i$ each completely belong to one $P_j$ (may be
distinct for different $V_i$). This comes in handy when it comes to finding
absorbers.

The property that we use most frequently is that the reduced graph in a way
inherits degree properties from its underlying graph $G$. This is nothing fancy
and should come as no surprise to anyone familiar with the (graph) regularity
method. Again, very conveniently, one can make it such that degrees are
`controlled' within certain predetermined sets, and not only globally in the
whole graph $\cR$. The following statement is almost a one-to-one copy of
\cite[Lemma~4.7]{ferber2022dirac} (slightly paraphrased for convenience).

\begin{lemma}[{\cite[Lemma~4.7]{ferber2022dirac}}]\label{lem:reg-lem-partition-min-degree}
  For all $h, t_0 \in \N$ and $\eps, \delta, \lambda > 0$, there exist $\eta, b,
  T > 0$ such that the following holds. Let $p \in [0, 1]$ and let $G$ be a
  sufficiently large $n$-vertex $(\eta, b, p)$-upper-uniform $3$-graph. Let
  $n_1, \dotsc, n_h \geq \lambda n$ and $P_1,\dotsc,P_h$ be a partition of
  $V(G)$ with $|P_i| = n_i$. Then there exists an $(\eps,p)$-regular partition
  $(V_i)_{i \in [t]}$ of $V(G)$, for some $t \in [t_0, T]$ and a corresponding
  reduced $t$-vertex $3$-graph $\cR = \cR((V_i)_{i \in [t]}, \eps, p, 2\eps)$
  with the following property. Let $\cP_i$ be the set of clusters $V_j$
  contained entirely in $P_i$ and let $t_i = |\cP_i|$. Then:
  \begin{enumerate}[(i)]
    \item $t_i \geq (\frac{n_i}{n} - \eps h)t$ for every $i \in [h]$.
    \item Let $d \in \{1,2\}$. Suppose that for some $i, j \in [h]$ all but at
      most $o(n^d)$ of the $d$-sets $S \subseteq P_i$ satisfy
      \[
        \deg_G(S, P_j) \geq \delta\binom{n_j-d}{3-d}p.
      \]
      Then all but at most $\sqrt\eps\binom{t}{d}$ of the $d$-sets $\cS
      \subseteq \cP_i$ satisfy
      \[
        \deg_\cR(\cS, \cP_j) \geq (\delta - (h + 2)\eps - \sqrt\eps -
        3/t_0)\binom{t_j-d}{3-d}.
      \]
  \end{enumerate}
\end{lemma}

We remark that, if one wants to only inherit minimum degree (that is, $d = 1$),
then standard double counting methods (see e.g.~\cite{noever2017local}) show
that this can be done without having the $\sqrt\eps t$ error term. Namely,
actually \emph{all vertices} satisfy the corresponding degree assumption. For $d
\geq 2$, the `almost all' is necessary.

Lastly, we use a hypergraph version of the infamous K{\L}R
conjecture\footnote{still known by this name, but has since its
 introduction~\cite{kohayakawa1997onk} been proven~\cite{balogh2015independent,
saxton2015hypergraph}; for a very recent, short, and
self-contained proof see~\cite{nenadov2022new}.} whose proof for \emph{linear
hypergraphs} was explicitly spelled out in \cite{ferber2020almost} but already
observed to hold in the work of Conlon, Gowers, Samotij, and
Schacht~\cite{conlon2014klr}. For a $3$-graph $H$ on vertex set $\{1, \dotsc,
t\}$ we denote by $\cG(H, n, m, p, \eps)$ the class of graphs $G$ obtained in
the following way. The vertex set of $G$ is a disjoint union $V_1 \cup \dotsb
\cup V_t$ of sets of size $n$. For each edge $ijk \in E(H)$ we add to $G$ an
$(\eps,p)$-regular $3$-graph with $m$ edges between the triple $(V_i,V_j,V_k)$
(and these are the only edges of $G$). A \emph{canonical copy} of $H$ in $G$ is
a $t$-tuple $(v_1,\dotsc,v_t)$ with $v_i \in V_i$ for every $i \in V(H)$ and
$v_iv_jv_k \in E(G)$ for every $ijk \in E(H)$. We write $G(H)$ for the number of
canonical copies of $H$ in $G$. Lastly, we need the notion of $3$-density
$m_3(H)$ of a $3$-graph $H$, which is defined as
\[
  m_3(H) = \max\Big\{ \frac{e(H') - 1 }{v(H') - 3} : H' \subseteq H \text{ with
  } v(H') \geq 4 \Big\}.
\]

\begin{theorem}\label{thm:KLR}
  For every linear $3$-graph $H$ and every $\alpha > 0$, there exist
  $\eps, \xi > 0$ with the following property. For every $\eta > 0$, there is a
  $C > 0$ such that if $p \geq CN^{-1/m_3(H)}$, then with probability $1 -
  e^{-\Omega(N^3p)}$ the following holds in $\cH^3(N, p)$. For every $n \geq
  \eta N$, $m \geq \alpha pn^3$, and every subgraph $G$ of $\cH^3(N,p)$ in
  $\cG(H, n, m, p, \eps)$, we have $G(H) \geq \xi n^{v(H)}p^{e(H)}$.
\end{theorem}

Strictly speaking, the subgraph $G$ that we later apply Theorem~\ref{thm:KLR} to
is not a member of $\cG(H, n, m, p, \eps)$, in particular not all
$(\eps,p)$-regular 3-graphs have exactly $m$ edges---the number of edges across
these are within a constant factor of each other. However, subsampling (see,
e.g.~\cite[Lemma 4.3]{gerke2005sparse}) circumvents this. For clarity of
presentation, we prefer to not explicitly spell out this argument.

\section{Covering random hypergraphs by loose paths}\label{sec:path-cover}

In this section we show a vital part of every strategy relying on the absorbing
method which in our specific problem reads as: the majority of vertices of $G
\subseteq \Hnp$ can be covered by a few (loose) paths. For the purposes of this
lemma, we consider single vertices to be loose \emph{paths of length zero}.

\begin{lemma}\label{lem:almost-path-cover}
  Let $d \in \{1,2\}$. For every $\gamma, \rho > 0$, there exists a $C > 0$ such
  that w.h.p.\ $\Hnp$ has the following property, provided that $p \geq C\log n
  \cdot \max\{n^{-3/2},n^{-3+d}\}$. Let $G \subseteq \Hnp$ with $v(G)\geq n/2$
  in which all but $o(n^{d})$ $d$-sets $S \subseteq V(G)$ satisfy $\deg_G(S)
  \geq (\delta_d+\gamma) p \binom{n-d}{3-d}$. Then there exist at most $\rho n$
  vertex-disjoint loose paths that cover the vertex set of $G$.
\end{lemma}

To a reader familiar with the topic, there is no magic that happens here. We
apply the sparse regularity lemma to $G$, find a desirable structure in the
obtained reduced graph, and then use it as a guide to construct loose paths in
$G$ itself. Perhaps the most innovative thing comes in the part where, in an
$(\eps,p)$-regular triple $(V_1,V_2,V_3)$, we find a long loose path covering
all but $o(|V_i|)$ vertices in each $V_i$. This itself relies on a widely-used
\emph{Depth-First Search} (DFS) technique employed explicitly in the context of
graphs in \cite{ben2012long, ben2012size}. (For an extremely neat application of
the method and more in-depth discussion see~\cite{krivelevich2016long}.)

\begin{lemma}\label{lem:covering-regular-triples}
  Let $\cH = (V_1,V_2,V_3;E)$ be a $3$-partite $3$-graph with $|V_1| = |V_3|= t$
  and $|V_2| = 2t$ such that for every choice $X_i \subseteq V_i$ of size $|X_i|
  = k$, there exists an edge in $\cH[X_1 \cup X_2 \cup X_3]$. Then there is a
  loose $uv$-path of length $2t - 4k$ in $\cH$ with $u,v \in V_1 \cup V_3$,
  and whose every other degree-one vertex lies in $V_2$.
\end{lemma}
\begin{proof}
  We explore the graph $\cH$ by using a variant of the \emph{Depth-First Search}
  (DFS) procedure as follows. To start with, set $T = \varnothing$, $S =
  \varnothing$, and $U = V(\cH)$. For as long as $U \cap (V_1 \cup V_3) \neq
  \varnothing$, do:
  \begin{itemize}
    \item if $S = \varnothing$, pick an arbitrary vertex from $U \cap (V_1 \cup
      V_3)$ and add it to $S$;
    \item if $S \neq \varnothing$ and there is an edge $uvw \in \cH$ such that
      $u \in V_i$ is the last added vertex to $S$, and $v \in U \cap V_2$ and $w
      \in U \cap V_{4-i}$, move $v, w$ from $U$ to $S$ in that order;
    \item otherwise, for $uvw \in \cH$ being the last three vertices added to
      $S$, move $v, w$ from $S$ to $T$ (if there is no such edge just move the
      only vertex $u$ in $S$ to $T$).
  \end{itemize}

  Observe that the vertices in $S$ at all times span a loose path as wanted in
  the lemma, but maybe not of the required length, and, crucially, there is never a
  time when some edge
   $uvw \in \cH$ is such that $u \in T \cap V_i$, $v \in U \cap V_2$, and $w
  \in U \cap V_{4-i}$, for $i \in \{1,3\}$. We aim to show that at some point
  we have $|S \cap V_2| \geq 2t - 4k + 1$ which is sufficient for the lemma to
  hold.

  In every step of the procedure either one (if $S = \varnothing$, say) or two
  vertices get moved from $U$ to $S$ or from $S$ to $T$. Consider the first
  moment in time when $|T \cap V_i| = k$, for some $i \in \{1, 3\}$; we may
  safely assume this happens for $V_1$. So, $|T \cap V_3| < k$, and moreover by
  the description of the procedure above, at this point we necessarily have $|T
  \cap V_2| < 2k$.

  As $|T \cap V_1| = k$, it cannot be that $|U \cap V_2|, |U \cap V_3| \geq k$,
  by the property of the lemma that such three sets must contain an edge.
  Assuming $|U \cap V_2| < k$, we have $|S \cap V_2| \geq 2t - 3k + 2$ as
  desired.

  Otherwise, suppose $|U \cap V_2| \geq k$ and $|U \cap V_3| < k$. From the fact
  that $|S \cap V_1| = |S \cap V_3| \pm 1$ we conclude that the intersections of
  $V_1$ and $V_3$ with $T \cup U$ are the same (up to $\pm 1$). This further
  implies $|U \cap V_1| < k$ (recall, $|T \cap V_3| < |T \cap V_1| = k$).
  Putting things together, we have $|S \cap (V_1 \cup V_3)| \geq 2t - 4k + 3$,
  and hence again $|S \cap V_2| \geq 2t - 4k + 2$ as desired.
\end{proof}

As a reminder, the exact values of $\delta_d$ are known for $3$-uniform
hypergraphs to be $\delta_1 = 7/16$ and $\delta_2 = 1/4$.

\begin{proof}[Proof of Lemma~\ref{lem:almost-path-cover}]
  Take $s \in \N$ sufficiently large and
  $\eps$ sufficiently small,
   in particular so that
  $\eps < \min\{\gamma^2/64, \rho/16\}$
  and $\Lambda :=
  \binom{s}{d}(\sqrt{\eps} + e^{-\Omega(\gamma^2 s)})$ is small
  enough, namely $\Lambda \leq \rho/100$.
  Pick $t_0$ large enough so that $t_0 \geq s / \Lambda $.
  Let $\eta =
  \eta_{\ref{lem:reg-lem-partition-min-degree}}(\eps, \delta_d+\gamma, t_0)$,
  $b = b_{\ref{lem:reg-lem-partition-min-degree}}(\eps, \delta_d+\gamma, t_0)$,
  and $T = T_{\ref{lem:reg-lem-partition-min-degree}}(\eps, \delta_d+\gamma,
  t_0)$.

  Since $p = \omega(n^{-2}\log{n})$, by Lemma~\ref{lem:hnp-is-upper-uniform}
  $\Hnp$ is w.h.p.\ $(\eta/2,b,p)$-upper-uniform, and so $G$ is
  $(\eta,b,p)$-upper-uniform then. Therefore, we can apply
  Lemma~\ref{lem:reg-lem-partition-min-degree} to $G$ with $\delta_d + \gamma$
  (as $\delta$), to obtain an $(\eps,p)$-regular partition $(V_i)_{i \in [t]}$
  for some $t_0 \leq t \leq T$ and a corresponding reduced graph $\cR =
  \cR((V_i)_{i \in [t]}, \eps, p, 2\eps)$. In particular, we get that for all
  but at most $\sqrt{\eps}\binom{t}{d}$ $d$-sets $\cS \subseteq V(\cR)$, it
  holds that $\deg_{\cR}(\cS) \geq (\delta_d + \gamma/2)\binom{t-d}{3-d}$. Let
  us remove from each $V_i$ at most two vertices to get that they are all of
  even size $\floor{n/t} - 1 \leq m \leq \ceil{n/t}$ (with slight abuse of
  notation we still refer to them as $V_i$).

  Let $R_1, \dotsc, R_{t/s}$ be a partition of $V(\cR)$ into disjoint subsets of
  size $s$, chosen uniformly at random. Recall that $\Lambda :=
  \binom{s}{d}(\sqrt{\eps} + e^{-\Omega(\gamma^2 s)})$. By
  Lemma~\ref{lem:random-set-almost-all-degree-inh}, with positive probability
  all but a $\Lambda$-fraction of the subgraphs $\cR_i := \cR[R_i]$ have minimum
  $d$-degree at least $(\delta_d+\gamma/4)\binom{s-d}{3-d}$. As we have picked
  $s$ to be even and large, each $\cR_i$ with this minimum degree has a loose
  Hamilton cycle by Theorem~\ref{thm:dense-graph}. The vertices not covered by
  loose cycles in $\mathcal{R}$ are then at most $\Lambda t + s$. We next show
  how to cover each loose cycle in $\mathcal{R}$ with not too many loose paths
  in $G$.

  Consider a (loose) Hamilton cycle in some $\cR_i$ and suppose without loss of
  generality that the clusters corresponding to the vertices of $\mathcal{R}_i$
  are $V_1, \dotsc, V_s$ in the order in which they appear on the cycle. Then we
  know that $(V_i, V_{i+1}, V_{i+2})$ is $(\eps,p)$-regular with density at
  least $2\eps p$, for every $i \in [s-1]$, $i$ odd ($s + 1$ is identified with
  $1$). Split every $V_i$, for $i$ odd, arbitrarily into $V_i = V_i^1 \cup
  V_i^2$, each of size $m/2$. By the definition of a regular triple, for all
  $X_i \subseteq V_i^1$, $X_{i+1} \subseteq V_{i+1}$, and $X_{i+2} \subseteq
  V_{i+2}^2$, each of size $\eps m$, we have
  \[
    d_G(X_i, X_{i+1}, X_{i+2}) \geq d_G(V_i, V_{i+1}, V_{i+2}) - \eps p \geq
    \eps p > 0,
  \]
  and in particular $e_G(X_i, X_{i+1}, X_{i+2}) > 0$.
  Lemma~\ref{lem:covering-regular-triples} implies there is a loose $uv$-path in
  $G[V_i^1, V_{i+1}, V_{i+2}^2]$ with $u,v \in V_i^1 \cup V_{i+2}^2$,
  which is of
  length $2 \cdot m/2 - 4\eps m = (1-4\eps)m$ and thus uses all but at most
  $8\eps m$ vertices in $V_i^1 \cup V_{i+1} \cup V_{i+2}^2$.

  Repeating this for every odd $i \in [s]$, we get $s/2$ loose paths covering
  all but at most $8\eps m \cdot s/2$ vertices in $V_1, \dotsc, V_s$. The whole
  thing can independently be repeated for every $\cR_i$ as well. In total, and
  counting vertices as paths of length zero, we found at most
  \[
    t/s \cdot s/2 \cdot (1 + 8\eps m) + (\Lambda t + s) \cdot m + 2t \leq t
    \cdot 8\eps n/t + 2 \Lambda n + o(n/t) \leq \rho n
  \]
  loose paths that cover the vertex set of $G$.
\end{proof}

\section{Connecting Lemma}\label{sec:connecting-lemma}

In this section we prove a vital ingredient both for constructing absorbers and
independently as a part of the proof of Theorem~\ref{thm:main-theorem}. Roughly
speaking, we show that one can connect prescribed pairs of vertices with short
paths through a set of vertices under some degree assumptions. Given a set $W
\subseteq V(G)$ in a graph $G$, an integer $\ell$, and a set of pairs
$\{x_i,y_i\}_i$ from $V(G) \setminus W$, a
\emph{$(\{x_i,y_i\}_i,W,\ell)$-matching} is a collection of internally
vertex-disjoint paths $P_i$, where each $P_i$ is of length at most $\ell$, has
$x_i, y_i$ as endpoints, and its remaining vertices belong to $W$.

\begin{lemma}[Connecting Lemma]\label{lem:connecting-lemma}
  Let $d \in \{1,2\}$. For every $\gamma, \xi > 0$ there exist $\eps, C > 0$
  such that w.h.p.\ $\Hnp$ satisfies the following, provided $p \geq C\log n
  \cdot \max\{n^{-3/2}, n^{-3+d}\}$. Let $G \subseteq \Hnp$ and $U, W \subseteq
  V(G)$ be disjoint subsets such that
  \begin{itemize}
    \item $|W| \geq \xi n$,
    \item all $d$-sets $S \subseteq U$, $S \subseteq W$ have $\deg_G(S, W) \geq
      (\delta_d+\gamma)p\binom{|W|-d}{3-d}$.
  \end{itemize}
  Then, for every family of distinct pairs $\{x_i,y_i\}_{i \in [t]}$ in $U$ such
  that $t \leq \eps|W|$ and every $u \in U$ appears in at most two pairs, there
  exists a $(\{x_i,y_i\}_i,W,4)$-matching in $G$.
\end{lemma}
\begin{proof}
  Given $\xi$ and $\gamma$, let $\tilde\gamma, \eps > 0$ be sufficiently small
  and $C > 0$ sufficiently large for the arguments below to go through.
  Condition on $\Hnp$ having the properties of
  Lemma~\ref{lem:1-edge-concentration}, Lemma~\ref{lem:2-edge-concentration},
  Lemma~\ref{lem:general-edge-concentration},
  Lemma~\ref{lem:large-second-neighbourhood} for $\tilde\gamma$ (as $\gamma$),
  and Lemma~\ref{lem:random-set-degree-inh}, which happens with high
  probability.

  Consider an auxiliary hypergraph $\cH$ (an $8$-graph in case $d=1$ and a
  $2$-graph in case $d=2$) with vertex set $[t] \cup W$ in which for $Z
  \subseteq W$, with $|Z| = 7$ if $d=1$ and $|Z|=1$ if $d=2$, an edge $\{i\}
  \cup Z$ exists if and only if there is a $x_iy_i$-path in $G$ all whose
  internal vertices belong to $Z$. Hence, a $[t]$-saturating matching in $\cH$
  corresponds to a $(\{x_i,y_i\}_i, W, 4)$-matching in $G$. Our plan is to use
  Haxell's matching theorem (Theorem~\ref{haxell}) to show this graph contains
  such a matching. It is sufficient to, for every $\cI \subseteq [t]$ and $Q
  \subseteq W$ of size $|Q| \leq 16|\cI|$, find an $i \in \cI$ and
  an $x_iy_i$-path in $G$ whose internal vertices all belong to $W \setminus Q$.
  Indeed, this implies there is an edge in $\cH$ that intersects $\cI$ but does
  not intersect $Q$, and so the condition in Theorem~\ref{haxell} is satisfied.
  We treat the cases $d=1$ and $d=2$ separately, starting with $d=2$.

  \paragraph{Codegree, $d = 2$.} Recall, $\delta_2 = 1/4$. Consider any $\cI
  \subseteq [t]$ and $Q \subseteq W$ of size $|Q| \leq 16|\cI|$. Let $\cP$ be a
  maximal set of pairs $\{x_i,y_i\}$ with $i \in \cI$ and all vertices distinct
  and assume towards a contradiction there is no $x_iy_i$-path of length at most
  four in $G$ with all internal vertices belonging to $W \setminus Q$. By the
  assumption of the lemma we have $|\cP| \geq |\cI|/3$. If $|\cP| \leq
  \eps^{-3}\log n/p$, using on the one hand the minimum codegree assumption and
  on the other the property of
  Lemma~\ref{lem:2-edge-concentration}~\ref{lem:2-edge-concentration-p1}, we
  obtain
  \[
    |\cP| \cdot (1/4+\gamma)p(|W|-2) \leq e_G(\cP, Q) \leq 48 \eps^{-4}|\cP|\log n,
  \]
  which is a contradiction for $C$ large enough, as $|W|p/8 \geq (\xi/8) \cdot
  C\log n \geq 48\eps^{-4}\log n$. Otherwise, if $|\cP| \geq \eps^{-3}\log n/p$,
  then by the property of
  Lemma~\ref{lem:2-edge-concentration}~\ref{lem:2-edge-concentration-p2},
  \[
    |\cP| \cdot (1/4+\gamma)p(|W|-2) \leq e_G(\cP, Q) \leq (1+\eps) 48|\cP|^2p,
  \]
  by taking a superset of $Q$ of size $16|\cI| \leq 48|\cP|$, leading to a
  contradiction as $|W|/8 \geq t/(8\eps) > 96|\cP|$.

  \paragraph{Degree, $d = 1$.} Recall, $\delta_1 = 7/16$. Let $W = W_1 \cup W_2
  \cup W_3$ be an equipartition of $W$ in which each $v \in U \cup W$, for
  $\tilde s = |W|/3$, satisfies
  \begin{equation}\label{eq:conn-min-deg}
    \deg_G(v, W_i) \geq (7/16+\gamma/2)p\binom{\tilde s - 1}{2}
  \end{equation}
  for all $i$. From the assumptions of the lemma, it is straightforward to show
  that such a partition exists by making use of the property of
  Lemma~\ref{lem:random-set-degree-inh} and the union bound.

  Consider any $\cI \subseteq [t]$ and $Q \subseteq W$ of size $|Q| \leq
  16|\cI|$. Assume that $\cI$ contains indices $i$ so that the collection
  $\{x_i, y_i\}$ consists of distinct vertices (a maximal subset of such $i$
  comprises at least a third of the original $\cI$ which has no influence on the
  proof). The $x_iy_i$-path we are trying to find is going to be such that it
  intersects each $W_1, W_3$ in exactly two vertices, and $W_2$ in exactly three
  vertices.

  Let $\tilde W_i := W_i \setminus Q$. It is sufficient to show that there is an
  index $i \in \cI$ for which:
  \stepcounter{propcnt}
  \begin{alphenum}
    \item\label{conn-exp-1} There exist $S_1 \subseteq N_G(x_i, \tilde W_1)$ and
      $S_3 \subseteq N_G(y_i, \tilde W_3)$, both of size $\sqrt{n}$;
    \item\label{conn-exp-2} Let $G_1$ and $G_3$ be auxiliary $2$-graphs on the
      same vertex set $\tilde W_2$ and $uv \in E(G_i)$ if $\{u, v\}$ form an
      edge in $G$ with a vertex from $S_i$, $i \in \{1, 3\}$. Then for each $i
      \in \{1, 3\}$
      \[
        \big|\{v \in \tilde W_2 : \deg_{G_i}(v) \geq 2\}\big| \geq
        \frac{\sqrt{7}}{4}\tilde s.
      \]
  \end{alphenum}

  Before proving these statements, let us show how to complete the proof. If we
  were to find three vertices $u, v, w \in \tilde W_2$ so that $e_G(S_1, u, v) >
  0$ and $e_G(S_3, w, v) > 0$, this would close a $x_iy_i$-path as desired. A
  path of length two in $G_1 \cup G_3$ with one edge in each of $G_1, G_3$
  corresponds exactly to a triple $u, v, w \in \tilde W_2$ as above. By
  \ref{conn-exp-2} and the pigeonhole principle, there must be a $v \in \tilde
  W_2$ with $\deg_{G_1}(v), \deg_{G_3}(v) \geq 2$. This implies we can choose
  $u$ as one of $v$'s neighbours in $G_1$ and $w$ as one of $v$'s neighbours in
  $G_3$ such that $u \neq w$.

  We first show there is an $i \in \cI$ for which \ref{conn-exp-1} holds. For
  this, it is sufficient to show that more than half of the indices $i \in \cI$
  are such that $\deg_G(x_i, \tilde W_1) \geq \tilde\gamma p\binom{n-1}{2}$. The
  existence of $i \in \cI$ and the corresponding sets $S_1$ and $S_3$ as in
  \ref{conn-exp-1} follows from the pigeonhole principle and the property of
  Lemma~\ref{lem:large-second-neighbourhood}.

  Let $X$ be a set of vertices $x_i$, $i \in \cI$, with $\deg_G(x_i, \tilde W_1)
  < \tilde\gamma p\binom{n-1}{2}$ and suppose $|X| \geq |\cI|/2$. Owing to the
  assumption on the minimum degree \eqref{eq:conn-min-deg} this means
  \begin{equation}\label{eq:conn-step-1-ub}
    e_G(X, Q, V(G)) \geq |X|\Big((7/16+\gamma/2)p\binom{\tilde s - 1}{2} -
    \tilde\gamma p\binom{n-1}{2}\Big) \geq 0.001\xi^2|\cI|n^2p.
  \end{equation}

  If $|Q|, |X| \leq \eps^{-3}\log n/(np)$, then the property of
  Lemma~\ref{lem:1-edge-concentration}~\ref{lem:1-edge-concentration-p1} implies
  \[
    e_G(X, Q, V(G)) \leq \eps^{-4} \max\{|X|, |Q|\} \log n = o(|\cI|n^2p).
  \]
  Otherwise, if $\max\{|Q|, |X|\} \geq \eps^{-3}\log n/(np)$ then by the
  property of
  Lemma~\ref{lem:1-edge-concentration}~\ref{lem:1-edge-concentration-p2}
  \[
    e_G(X, Q, V(G)) \leq 2\max\{|X||Q|np,48\eps^{-3}|\cI|\log n\},
  \]
  where we take a superset of $X$ or $Q$ of size exactly $\eps^{-3}\log n/(np)$
  if necessary. This again leads to a contradiction with
  \eqref{eq:conn-step-1-ub} since $|X||Q| \leq 48\eps|\cI|n$ and by our choice
  of $\eps$ (with room to spare).

  For \ref{conn-exp-2}, let
  \[
    \cP:= E(G_1), \quad S_2 := \{ u \in \tilde W_2 : \deg_{G_1}(u) \geq 2 \},
    \quad \text{and} \quad \cP_1 := \big\{ \{u,v\} \in \cP : u \notin S_2
    \big\}.
  \]
  Assume for contradiction $|S_2| < \sqrt{7}\tilde s/4$. Note that, by the
  property of Lemma~\ref{lem:general-edge-concentration} and taking a superset
  of $Q$ of size $48\eps n$ if necessary,
  \begin{equation}\label{eq:conn-general-conc}
    e_G(S_1, Q, V(G)) \leq 48(1+\eps)|S_1|\eps n^2p.
  \end{equation}
  By \eqref{eq:conn-min-deg} and for $\eps$ sufficiently small, we thus get
  \begin{equation}
    \label{eq:conn-last-step-lower-bound}
    e_G(S_1, \cP) \osref{\eqref{eq:conn-min-deg}}\geq
    |S_1|(7/16+\gamma/2)p\binom{\tilde s-1}{2} - e_G(S_1, Q, V(G))
    \osref{\eqref{eq:conn-general-conc}}\geq |S_1|(7/16+\gamma/4)p\binom{\tilde
    s}{2}.
  \end{equation}
  On the other hand, noting that $|\cP_1| \leq n$, by the property of
  Lemma~\ref{lem:2-edge-concentration} and taking supersets of $\cP_1$ and $S_1$
  if necessary,
  \[
    e_G(S_1, \cP_1) \leq
    \max\{\eps^{-4}|\cP_1|\log n, (1+\eps)n^2p \} = o(|S_1|p\tilde s^2)
  \]
  and so by the property of Lemma~\ref{lem:general-edge-concentration} again
  \[
    e_G(S_1, \cP) = e_G\Big(S_1, \binom{S_2}{2}\Big) + e_G(S_1, \cP_1) \leq
    (1+\eps)|S_1|p \binom{\sqrt{7}\tilde s/4}{2} + o(|S_1|p\tilde s^2),
  \]
  which is smaller than $|S_1|(7/16+\gamma/4)p\binom{\tilde s}{2}$ by our choice
  of $\eps$, thereby contradicting \eqref{eq:conn-last-step-lower-bound}.
\end{proof}

\section{The absorbing method in sparse hypergraphs}\label{sec:absorbing-method}

The absorbing method was initially explicitly introduced by R\"{o}dl,
Ruci\'{n}ski, and Szemer\'{e}di~\cite{rodl2006dirac} even though the implicit
idea has its roots in the works of Krivelevich~\cite{krivelevich1997triangle}
and Erd\H{o}s, Gy\'{a}rf\'{a}s, and Pyber~\cite{erdHos1991vertex}. Recently it
has seen a surge of interest and has been used in a variety of settings:
combinatorial designs~\cite{glock2023existence, keevash2014existence},
decompositions~\cite{glock2021resolution, kuhn2013hamilton}, Steiner
systems~\cite{kwan2022high, ferber2020almost}, Ramsey
theory~\cite{bucic2021tight, keevash2021cycle}, colouring
(hyper)graphs~\cite{montgomery2021proof, kang2022proof},
embeddings~\cite{montgomery2019spanning, georgakopoulos2021spanning}, and many,
many more.

In principle, the idea behind it is simple. It relies on reducing the problem of
finding a spanning subgraph $\cS$ in some graph $G$ to the one of finding an
almost spanning subgraph $\cS' \subseteq \cS$, say, one of size
$(1-o(1))|V(G)|$. Often times, the latter is significantly easier to solve, if
nothing else, just for the fact that we have quite a bit of room for error. In
practice, the implementation of this idea typically depends on first embedding a
highly structured graph $A$ into $G$, which is capable of extending any partial
embedding to a complete one. The task of designing the graph $A$ with this
magical property is where the whole \emph{art of absorption} lies in. Usually,
it is specific to the embedding problem at hand and is where the main difficulties
arise---it can be, first, quite surprising that such a graph should even exist
and, second, challenging to find it in the host graph $G$ in a convenient way
(or in any way for that matter).

When dealing with paths, the `design' of the graph $A$ is not too complex. We
make this more rigorous.

\begin{definition}[$(a,b,R)$-absorber]\label{def:absorbing-graph}
  An $(a,b,R)$-absorber is a graph $A$ with $R \subseteq V(A)$, $a,b \in V(A)
  \setminus R$, with the property that for every $R' \subseteq R$ with $|R'| <
  |R|/2$ such that $V(A) \setminus R'$ has odd cardinality, there is a loose
  $ab$-path in $A$ with vertex set $V(A) \setminus R'$.
\end{definition}

The next lemma handles the second step of the method: actually finding the
graph $A$ in the subgraph $G$ of the random graph $\Hnp$.

\begin{lemma}[Absorbing Lemma]\label{lem:absorbing-lemma}
  Let $d \in \{1,2\}$. For every $\gamma, \xi > 0$ there exist $\alpha, C > 0$
  such that w.h.p.\ $\Hnp$ has the following property, provided that $p \geq
  C\log n \cdot \max\{n^{-3/2}, n^{-3+d}\}$. Let $G \subseteq \Hnp$ with
  $\delta_d(G) \geq (\delta_d+\gamma)p\binom{n-d}{3-d}$. Then, for a uniform
  random set $R \subseteq V(G)$ of size $\alpha n$, w.h.p.\ there exists an
  $(a,b,R)$-absorber in $G$ of order at most $\xi n$.
\end{lemma}

It may seem awkward that there are \emph{two probabilistic statements} in the
lemma. Indeed, the \emph{first w.h.p.} establishes some typical properties a
random $3$-uniform hypergraph has (such as, e.g., distribution of the edges),
whereas the \emph{second w.h.p.} is over the choice of $R$. Namely, once we
condition on $\Hnp$ having these `nice' properties, then for a randomly chosen
set there is an absorber with high probability. In fact, the lemma could be
written in a quantitative form, saying how w.h.p.\ in $\Hnp$ there are at least,
say, a $(1 - n^{-5})$-fraction of choices for $R$ which yield an absorber. We
found this a bit more cumbersome to deal with and decided to go with the former.

As per usual, the `construction' of such a graph $A$ consists of carefully
patching up many small structures.

\begin{definition}[$xy$-absorber]\label{def:absorber}
  An $xy$-\emph{absorber} is a graph that consists of:
  \begin{itemize}
    \item a path $P$, which we refer to as the \emph{covering path}, and
    \item a path $Q$, which we refer to as the \emph{non-covering path}, with
      $V(Q) = V(P) \setminus \{x,y\}$ and whose endpoints are identical with
      those of $P$.
  \end{itemize}
\end{definition}

We now define the $xy$-absorber we use, depending on $d$ (in doing this we draw
inspiration from \cite{buss2013minimum}).
Both of these, however, have a much more natural visual representation depicted
in Figure~\ref{fig:absorber-1} and Figure~\ref{fig:absorber-2}.

\begin{definition}
  The $xy$-absorber $A_{xy}^d$ is defined as:
  \begin{enumerate}[leftmargin=5em]
    \item[\textbf{($\mathbf{d = 2}$)}] consists of nine vertices $x, y, v_1,
      v_2, v_3, v_4, v_5, v_6, v_7$, and the edges $v_1v_2v_3$, $v_3v_4v_5$,
      $v_5v_6v_7$, $v_2xv_4$, $v_4yv_6$.
      \begin{figure}[!htbp]
        \centering
        \includegraphics[scale=1]{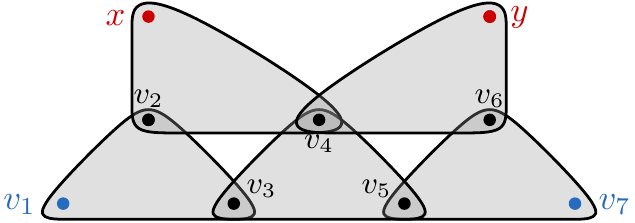}
        \caption{The $xy$-absorber $A_{xy}^2$.}
        \label{fig:absorber-2}
      \end{figure}
    \item[\textbf{($\mathbf{d = 1}$)}] consists of:
      \begin{itemize}[leftmargin=-2em]
        \item a cycle $x_1, x_2, x_3, x, x_4, \dotsc, x_{10}, P_x, x_1$, where
          $P_x$ is an $x_1x_{10}$-path of length four;
        \item a cycle $y_1, y_2, y_3, y, y_4, \dotsc, y_{10}, P_y, y_1$, where
          $P_y$ is a $y_1y_{10}$-path of length four;
        \item a copy of $A_{x_7y_7}^2$;
        \item edges $a_1a_2a_3$, $a_2a_4x_2$, $a_3a_4x_9$ and $b_1b_2b_3$,
          $b_2b_4y_2$, $b_3b_4y_9$;
        \item an $x_5y_5$-path and a $v_7b_1$-path, both of length four.
      \end{itemize}
      \begin{figure}[!htbp]
        \captionsetup[subfigure]{textfont=scriptsize}
        \centering
        \begin{subfigure}{.49\textwidth}
          \centering
          \includegraphics[width=1\linewidth]{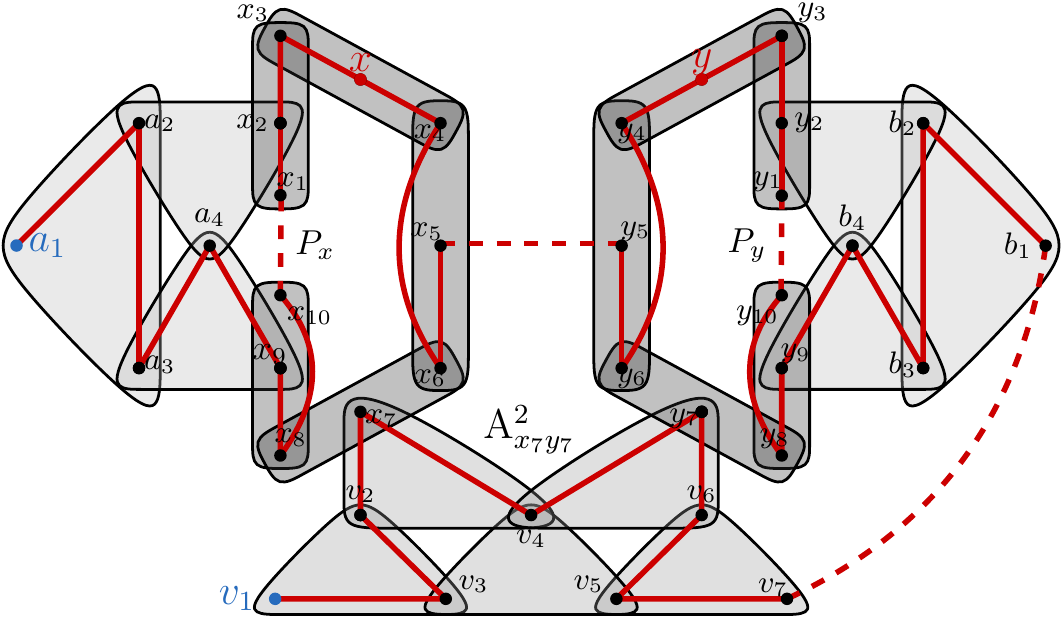}
          \caption{The covering $a_1v_1$-path}
        \end{subfigure}%
        \hfill
        \begin{subfigure}{.49\textwidth}
          \centering
          \includegraphics[width=1\linewidth]{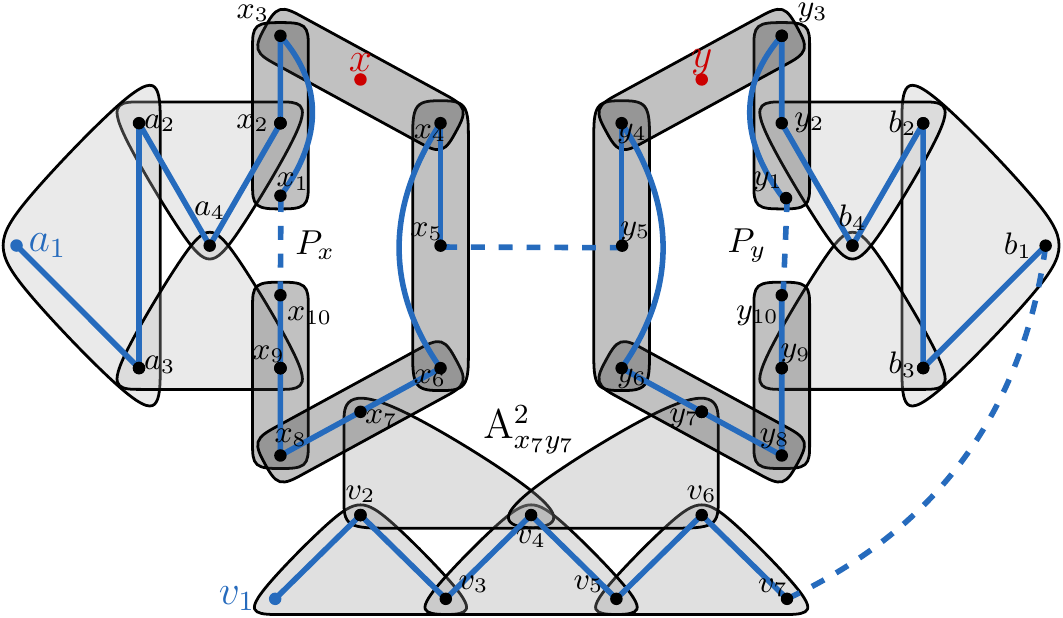}
          \caption{The non-covering $a_1v_1$-path}
        \end{subfigure}%
        \caption{The $xy$-absorber $A_{xy}^1$. The dashed lines represent paths
        of length four.}
        \label{fig:absorber-1}
      \end{figure}
  \end{enumerate}
\end{definition}

For future reference, the graph obtained by removing $P_x$, $P_y$, the
$x_5y_5$-path, and the $v_7b_1$-path is called a \emph{backbone} of the
$xy$-absorber $A_{xy}^1$ (see Figure~\ref{fig:absorber-backbone} below). We
refer to all vertices of $A^d_{xy}$ other than $x,y$ as its \emph{internal
vertices}.

To build an $(a,b,R)$-absorber, we plan to string together a number of copies of
$A_{xy}^d$. Clearly, we cannot just build disjoint $xy$-absorbers for every pair
of vertices $x, y \in R$ for the simple reason of there not being enough space
for that as $\binom{|R|}{2} \gg n$. A way of dealing with this originated in the
work of Montgomery~\cite{montgomery2019spanning}, who used an idea of looking at
an auxiliary \emph{bounded degree} graph, which serves as a template for which
pairs of vertices of $R$ to use. Roughly speaking, there is a graph $T = (R, E)$
on the vertex set $R$ and with $\Delta(T) = O(1)$, so that if for every $xy \in
E(T)$ we find disjoint $xy$-absorbers, then the obtained graph is an
$(a,b,R)$-absorber, for some $a, b \notin R$.

The following lemma is very similar to \cite[Lemma~7.3]{ferber2022dirac}, but it
has the additional property that the template graph it describes has bounded
maximum degree. The proof, being almost identical to that of the mentioned
lemma, is for completeness spelled out in the appendix.

\begin{lemma}\label{lem:template}
  There is an $L > 0$ such that the following holds. For any sufficiently large
  $m$, there exists a graph $T$ with $2m \leq v(T) \leq Lm$, maximum degree at
  most $L$, and a set $Z$ of $m$ vertices, such that for every $Z' \subseteq Z$
  with $|Z'| < m/2$ and $|V(T) \setminus Z'|$ even, the graph $T - Z'$ has a
  perfect matching.
\end{lemma}

At this point, we essentially reduced our goal to finding a single $xy$-absorber
for a prescribed pair of vertices $x, y \in R$. This is the key lemma of this
section.

\begin{lemma}\label{lem:single-absorber}
  Let $d \in \{1,2\}$. For any $\gamma > 0$ there exist $\lambda, C > 0$, such
  that w.h.p.\ $\Hnp$ satisfies the following, provided that $p \geq C\log n
  \cdot \max\{n^{-3/2}, n^{-3+d}\}$. Let $G \subseteq \Hnp$ with $\delta_d(G)
  \geq (\delta_d+\gamma)p\binom{n-d}{3-d}$ and let $S \subseteq V(G)$ with $|S|
  \leq \lambda n$. For any $x,y \in V(G) \setminus S$, that if $d = 1$
  additionally satisfy $\deg_{G-S}(x), \deg_{G-S}(y) \geq
  (\delta_d+\gamma/2)p\binom{n-1}{2}$, there exists a copy of $A_{xy}^{d}$ in
  $G-S$.
\end{lemma}

We remark that the degree assumption for $x, y$ is not needed in the $d = 2$
case. The proof of this lemma is the most intricate part of the argument so we
defer it to Section~\ref{sec:finding-single-absorber}. With all these
ingredients we can now prove the absorbing lemma.

\begin{proof}[Proof of Lemma~\ref{lem:absorbing-lemma}]
  Let $D$ denote the number of vertices in $A_{xy}^d$ and further set
  \begin{gather*}
    L = L_{\ref{lem:template}}, \quad
    \lambda = \min\{\lambda_{\ref{lem:varying-size-sets}}(\gamma/64, 8DL),
    \lambda_{\ref{lem:single-absorber}}(\gamma/2)\}, \quad \text{and} \quad
    \eps = \eps_{\ref{lem:connecting-lemma}}(\gamma/2).
  \end{gather*}
  Having chosen these, let $\alpha$ be sufficiently small with respect to all
  constants (in particular, small with respect to the given $\xi > 0$) and $C$
  sufficiently large. Write $m := \alpha n$, let $T$ be the auxiliary graph
  given by Lemma~\ref{lem:template} for $m$, and let $Z$ be the set of special
  vertices in $T$ as in the lemma. Assume $\Hnp$ is such that it satisfies the
  conclusion of Lemma~\ref{lem:varying-size-sets} with $\gamma/64$ (as $\gamma$)
  and $8DL$ (as $C$), Lemma~\ref{lem:random-set-degree-inh},
  Lemma~\ref{lem:connecting-lemma}, and Lemma~\ref{lem:single-absorber} with
  $\gamma/2$ (as $\gamma$).

  Choose a uniform random set $R \subseteq V(G)$ of size $\alpha n$ and let $R_0
  \cup V \cup W$ be a partition of $V(G) \setminus R$ such that:
  \stepcounter{propcnt}
  \begin{alphenum}
    \item\label{abs-sets-size} $|R| = m$, $|R_0| = v(T)-m$, and $|W| = n/100$;
    \item\label{abs-sets-deg} $\deg_G(S, X) \geq
      (\delta_d+\gamma/2)p\binom{|X|-d}{3-d}$, for every $d$-set $S \subseteq
      V(G)$ and $X \in \{R \cup R_0, V, W\}$.
  \end{alphenum}
  Note that, in particular, $|V| = (1 - 1/100 - o(1))n \geq n/2$. By the
  property of Lemma~\ref{lem:random-set-degree-inh} almost any choice of such a
  partition will do.

  Let $f \colon V(T) \to R_0 \cup R$ be a bijection which maps the vertices
  belonging to the set $Z$ in $T$ to $R$. For each edge $xy \in E(T)$, we plan
  to first find an $f(x)f(y)$-absorber $A_{f(x)f(y)}^d$ whose internal vertices
  belong to $G[V]$ in such a way that all these absorbers are pairwise internally
  vertex-disjoint. Subsequently, we string all the $f(x)f(y)$-absorbers together
  by an application of the Connecting Lemma (Lemma~\ref{lem:connecting-lemma})
  over the set $W$.

  The first part of this is to be completed using Haxell's condition and
  Lemma~\ref{lem:single-absorber}. Let $\cH$ be a $(D-1)$-uniform hypergraph
  with vertex set $U \cup V$, where $U = \big\{\{f(x), f(y)\} : xy \in
  E(T)\big\}$. For $\{u_x,u_y\} \in U$ and $S \subseteq V$ with $|S| = D-2$, we
  add the edge $\{u_x,u_y\} \cup S$ to $\cH$ if and only if there is a copy of
  $A_{u_xu_y}^d$ in $G[\{u_x,u_y\} \cup S]$. Then what we are looking for is
  precisely a $U$-saturating matching in $\cH$. Comparing to Haxell's condition
  (Lemma~\ref{haxell}), it is sufficient to show that for every $U' \subseteq U$
  and $V' \subseteq V$ with $|V'| \leq 2D|U'|$, there is some $\{u_x,u_y\} \in
  U'$ and a copy of $A_{u_xu_y}^d$ in $G$ whose internal vertices are fully
  contained in $V \setminus V'$.

  As $\Delta(T) \leq L$, we can greedily find a set of pairwise vertex-disjoint
  edges $xy \in E(T)$ for which $\{f(x),f(y)\} \in U'$ and which is of size at
  least $|U'|/(2L)$. For simplicity, we assume $U'$ already consists only of
  such disjoint edges---this changes nothing in the proof. In case $d = 2$,
  straightforwardly applying Lemma~\ref{lem:single-absorber} with any
  $f(x),f(y)$ for which $\{f(x),f(y)\} \in U'$ (as $x,y$), $G[V\cup
  \{f(x),f(y)\}]$ (as $G$), and $V'$ (as $S$), we are done. Note that we can
  indeed do so since
  \[
    |V'| \leq 4DL|U'| \leq 4DL \cdot e(T) \leq 4DL^3 m = 4DL^3 \alpha n \leq
    \lambda n.
  \]
  In the other case, $d = 1$, the only thing remaining is to check that there
  are $f(x),f(y)$ with $\{f(x),f(y)\} \in U'$ that satisfy the degree
  requirement $\deg_G(f(x), V \setminus V'), \deg_G(f(y), V \setminus V') \geq
  (\delta_1+\gamma/4)p\binom{|V|-1}{2}$. Towards contradiction, suppose there is
  no such $f(x), f(y)$. Let $Q \subseteq V(G)$ be the union of all $f(x)$ that
  violate the prior requirement and such that for some $f(y)$, $\{f(x), f(y)\}
  \in U'$, and assume $|Q| \geq |U'|/2$. On the one hand, this means
  \[
    e_G(Q, V', V) \osref{\ref{abs-sets-deg}}\geq |Q| \cdot
    (\gamma/4)p\binom{|V|-1}{2} \osref{\ref{abs-sets-size}}\geq
    (\gamma/64)|Q|n^2p,
  \]
  while on the other, by the property of Lemma~\ref{lem:varying-size-sets} since
  $|Q| \leq 2|U'| \leq \lambda n$ and $|V'| \leq 4DL|U'| \leq 8DL|Q|$,
  \[
    e_G(Q, V', V) \leq e_G(Q, V', V(G)) < (\gamma/64)|Q|p\binom{n-1}{2} \leq
    (\gamma/64)|Q|n^2p,
  \]
  which is a contradiction. So, $|Q| < |U'|/2$ and by the pigeonhole principle,
  there has to exist a pair $\{f(x),f(y)\} \in U'$ with the desired degree value
  into $V \setminus V'$.

  Finally, we use Lemma~\ref{lem:connecting-lemma} to connect all
  $f(x)f(y)$-absorbers into
  one loose path. Denote the absorbers we have found as $A_1, \dotsc, A_t$, for
  $t = e(T)$, each $A_i$ a copy of $A_{xy}^d$. We want to, for every $i \in
  [t-1]$, connect $v_1 \in V(A_i)$ and $a_1 \in V(A_{i+1})$ if $d=1$ (see
  Figure~\ref{fig:absorber-1}) and $v_7 \in V(A_i)$ and $v_1 \in V(A_{i+1})$ if
  $d=2$ (see Figure~\ref{fig:absorber-2}). Let $Y \subseteq V$ be the set
  of (the images of) all these $2(t-1)$ vertices we would like to connect. We
  use the property of Lemma~\ref{lem:connecting-lemma} with $W$, $Y$ (as $U$)
  which we can do by \ref{abs-sets-deg} and since $|W| \geq n/100$ and the total
  number of pairs we want to connect with paths is
  \[
    e(T) - 1 \leq L^2m = L^2 \alpha n \leq \eps n/100 \leq \eps|W|.
  \]
  Lastly, observe that the total number of vertices used by this whole procedure
  is
  \[
    e(T) \cdot D + (e(T) - 1) \cdot 7 \leq L^2 \alpha n \cdot (D+7) \leq \xi n,
  \]
  as desired.

  It remains to establish that the graph $A$ obtained by connecting the
  absorbers $A_1, \dots, A_t$ as described comprises an $(a,b,R)$-absorber,
  where $a$ and $b$ are (the images of) $a_1 \in V(A_1)$ and $v_1 \in V(A_t)$ if
  $d = 1$ and (the images of) $v_1 \in V(A_1)$ and $v_7 \in V(A_t)$ if $d = 2$.
  Namely, let $R' \subseteq R$ of size $|R'| < |R|/2$ be such that $V(A)
  \setminus R'$ has odd cardinality. Note that $V(A) \setminus (R_0 \cup R)$
  must also be of odd cardinality as it can be covered by a loose path (by
  taking the non-covering paths of all individual $f(x)f(y)$-absorbers) and
  hence, crucially, $R_0 \cup (R \setminus R')$ is of even cardinality. In
  particular, this means there is a perfect matching for $f^{-1}(R_0 \cup (R
  \setminus R'))$ in $T$. Consider the set of $f(x)f(y)$-absorbers corresponding
  to the edges in this matching. Then, as a witness for the absorbing property
  of $A$, we can use the covering path for all these absorbers, the non-covering
  path for all other $f(x)f(y)$-absorbers, and the short paths connecting the
  absorbers to get the desired loose path between $a$ and $b$.
\end{proof}

\subsection{Finding an $xy$-absorber robustly}\label{sec:finding-single-absorber}

In this section we prove Lemma~\ref{lem:single-absorber}. The case $d=2$ is much
simpler, so we deal with it first.

\begin{proof}[Proof of Lemma~\ref{lem:single-absorber}, case $d=2$]
  Recall, $\delta_2 = 1/4$. Let $\lambda =
  \lambda_{\ref{lem:robust-codegree-inh}}(\gamma,1/4)$ and suppose $\Hnp$
  satisfies the conclusion of Lemma~\ref{lem:general-edge-concentration} and
  Lemma~\ref{lem:robust-codegree-inh}.

  We use the property of Lemma~\ref{lem:robust-codegree-inh} to get a set $T
  \subseteq \binom{V(G) \setminus S}{2}$ of size $|T| \leq 10^9 \log n/p \leq
  10^{-6}n$ (by choosing $C$ sufficiently large) such that all pairs $\{u, v\}
  \in \binom{V(G) \setminus S}{2} \setminus T$ satisfy $\deg_F(u, v) \geq
  (1/4+\gamma/2)p(n-|S|-2) = \Omega(\log n)$, for $F := G-S$. Recall, our goal
  is to find a copy of $A_{xy}^2$ in $F$ (see Figure~\ref{fig:absorber-2}).

  Since $T$ is relatively small, there must be some $v_4 \in V(F) \setminus
  \{x,y\}$ such that neither $\{x,v_4\}$ nor $\{y, v_4\}$ belong to $T$. Having
  chosen $v_4$, take distinct $v_2, v_6 \in V(F) \setminus \{x,y,v_4\}$ such
  that $v_2xv_4, v_4yv_6$ are edges in $F$. Let
  \[
    V_3 := \{ v \in V(F) \setminus \{x,y,v_2,v_4,v_6\} : \{v, v_2\},\{v,v_4\}
    \notin T\}
  \]
  and
  \[
    \overline{V_5} := \{x, y, v_2, v_4, v_6\} \cup \{ v \in V(F) : \{v, v_6\}
    \in T \},
  \]
  and note that $|V_3| \geq (1-10^{-2})n$ and $|\overline{V_5}| \leq 10^{-5}n$.
  It is enough to show that there is an edge $v_3v_4v_5$ in $F$ such that $v_3
  \in V_3$ and $v_5 \in V(F) \setminus \overline{V_5}$, because $\{v_2,v_3\},
  \{v_5,v_6\} \notin T$ imply that we can choose $v_1$ and $v_7$ as desired for
  $A_{xy}^2$.

  Suppose for contradiction such an edge does not exist in $F$. Then, from the
  codegree assumption and the property of
  Lemma~\ref{lem:general-edge-concentration}, we have
  \[
    |V_3|/2 \cdot (1/4+\gamma/2)p(n-|S|-2) \leq e_F(v_4,V_3,\overline{V_5})
    \leq e_G(v_4,V_3,\overline{V_5}) \leq 2 \cdot 10^{-5}|V_3|np,
  \]
  which is a contradiction. Thus, an edge $v_3v_4v_5 \in F$ exists as desired.
\end{proof}

In what follows we provide a proof of Lemma~\ref{lem:single-absorber} for $d=1$.
This is the most involved part of the whole proof and we, for convenience of
reading, first give a brief outline. The main idea relies on an intricate
combination of the sparse regularity method and the connecting lemma. It
consists of two almost independent steps:
\begin{enumerate*}[(1)]
  \item find a backbone of an $A_{xy}^1$ absorber;
  \item use the Connecting Lemma (Lemma~\ref{lem:connecting-lemma}) to find the
    remaining loose paths which comprise an absorber $A_{xy}^1$ (see
    Figure~\ref{fig:absorber-1}).
\end{enumerate*}
Most of the difficulty lies in the first part. To do this, we use the
`contraction technique' of Ferber and Kwan~\cite{ferber2022dirac} and the
regularity method. In order for the next steps to make more sense, we first
introduce a definition.

\begin{definition}
  A \emph{contracted backbone of an absorber} is a graph that consists of:
  \begin{itemize}
    \item edges $x' x_7 x_8$, $x_8 x_9 x_{10}$ and $a_1 a_2 a_3$, $a_2 a_4 x'$,
      $a_3 a_4 x_9$;
    \item edges $y' y_7 y_8$, $y_8 y_9 y_{10}$ and $b_1 b_2 b_3$, $b_2 b_4 y'$,
      $b_3 b_4 y_9$;
    \item a copy of $A^2_{x_7 y_7}$.
  \end{itemize}
\end{definition}

A contracted backbone of an absorber can be thought of as starting with a
backbone of $A_{xy}^1$ and \emph{contracting} the edges $x_1x_2x_3,
x_3xx_4,x_4x_5x_6$ into a single vertex $x'$, and keeping only the edges `to the
outside' (that is, ones containing $x_2$ or $x_6$). The same is done to obtain
$y'$. Note that, in this context, the vertices $x'$ and $y'$ play a special role
and we often explicitly mention them when talking about the contracted backbone
of an absorber. Strictly speaking, a name for this structure that is more
descriptive of the aforementioned contraction operation might be a `contracted
backbone of an $xy$-absorber'. However, as the graph itself does not contain the
vertices $x,y$, we omit them from the name to avoid confusion. For a more
natural visual representation we depict this contraction operation below.

\begin{figure}[!htbp]
  \captionsetup[subfigure]{textfont=scriptsize}
  \centering
  \begin{subfigure}{.49\textwidth}
    \centering
    \includegraphics[width=\textwidth]{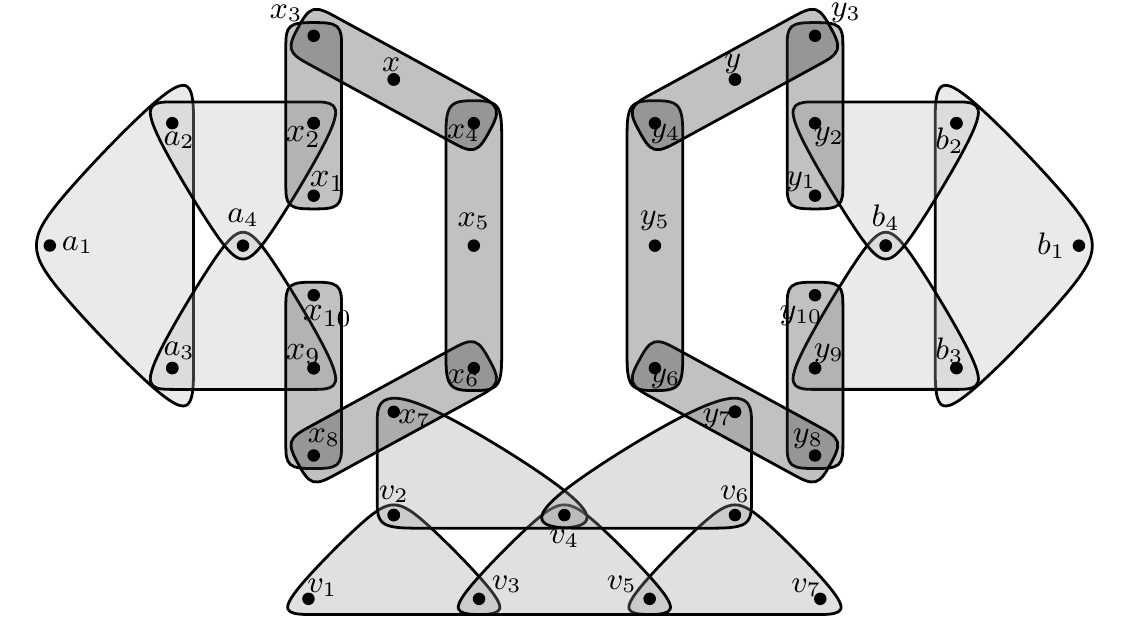}
    \caption{The backbone of an absorber $A_{xy}^1$.}
    \label{fig:absorber-backbone}
  \end{subfigure}%
  \hfill
  \begin{subfigure}{.49\textwidth}
    \centering
    \includegraphics[width=\textwidth]{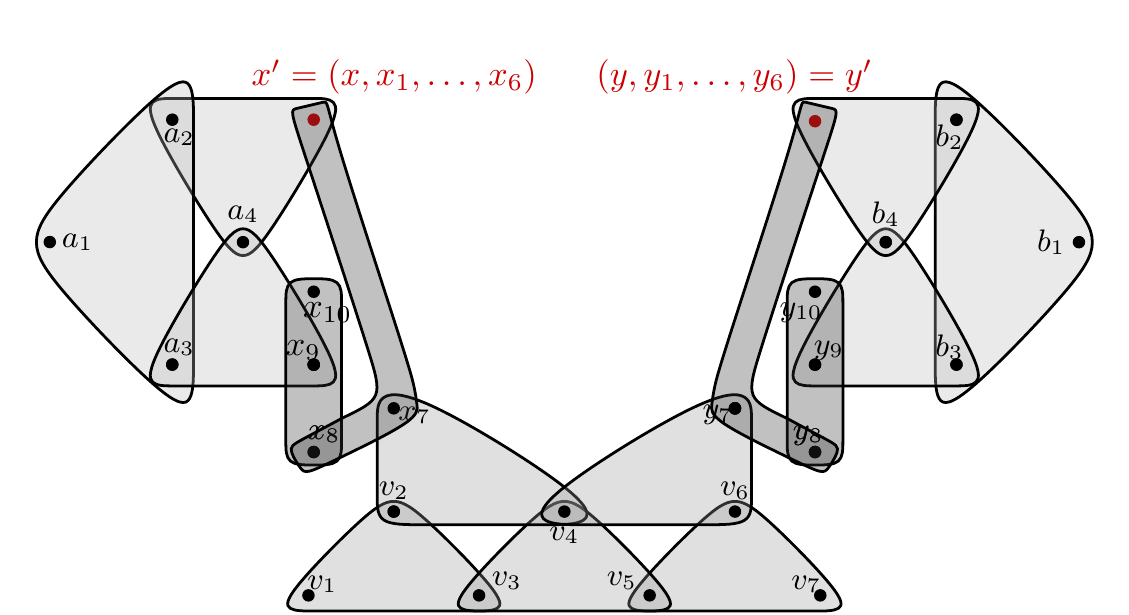}
    \caption{The contracted backbone of an absorber $A_{xy}^1$.}
    \label{fig:contracted-backbone}
  \end{subfigure}%
  \caption{The backbone and contracted backbone of an absorber $A^1_{xy}$.}
  \label{fig:absorber-without-paths}
\end{figure}

The contraction operation on $G \subseteq \Hnp$, almost analogous to the one of
Ferber and Kwan, is defined as follows. For a $3$-graph $G$, a collection $\cU$
of disjoint sets $U_1, U_2 \subseteq V(G)$, and a family of disjoint $4$-tuples
$\cF \subseteq V(G)$ we let $G(\cU, \cF)$ be a $3$-graph on vertex set $U_1 \cup
U_2 \cup \cF$ and whose edge set is given as follows: add first all edges from
$G[U_1 \cup U_2]$ and next, for every $\mbw = (w_1,w_2,w_3,w_4) \in \cF$, we
add an edge $\mbw u v$ to $G(\cU,\cF)$ if and only if $u, v \in U_1$ and
$w_2uv \in E(G)$ or $u, v \in U_2$ and $w_4uv \in E(G)$.

We can now continue with the outline. Namely, imagine for a moment that we can
find a collection of $\Omega(n)$ distinct $4$-tuples $(x_1,x_2,x_5,x_6)$ such
that, $x_1x_2x_3, x_3xx_4, x_4x_5x_6 \in E(G)$, for some $x_3, x_4$. Let us, for
every such tuple, contract these vertices (and edges) into a single vertex and
keep only edges to the outside that contain $x_2$ or $x_6$ (as in the contracted
graph $G(\cdot,\cdot)$ above). Denote the set of these new vertices as $X'$.
Now, do the same for $y$ while keeping all these disjoint to obtain $Y'$. If we
were to find a copy of the contracted backbone of an absorber with $x', y'$
mapped into some
vertices of $X'$ and $Y'$ respectively, we would be done by just undoing the
contraction operation.

To do this last part, we rely on the sparse regularity lemma
(Lemma~\ref{lem:reg-lem-partition-min-degree}) and Theorem~\ref{thm:KLR}. If all
the previous steps have been done carefully, what remains of the graph still
satisfies all the necessary conditions (in particular, the minimum degree will
be sufficiently large) to do this. First, we show that in the reduced graph
obtained from the application of Lemma~\ref{lem:reg-lem-partition-min-degree} we
can find a copy of the contracted backbone of an absorber with the vertices $x'$
and $y'$ being mapped to their corresponding `clusters' belonging to $X'$ and
$Y'$. Subsequently we use Theorem~\ref{thm:KLR} to transfer it into a canonical
copy of the same graph in $G$\footnote{Strictly speaking, this copy is found in
a contracted version of $G$ and then unfolded into a copy of $A_{xy}^1$ in $G$
itself.}. As we prove in the appendix (not to interrupt the flow of the main
argument), the contracted backbone of an absorber is just sparse enough to exist
in the regular partition.

\begin{claim}\label{cl:m3-density}
  The $m_3$ density of a contracted backbone of an absorber is $2/3$.
\end{claim}

There is one crucial difference compared to the method used in
\cite{ferber2022dirac}, reflected in the fact that we work with $p =
\Theta(n^{-3/2}\log n)$, in comparison to $p = \Omega(n^{-1})$. Namely, the
neighbourhood of a vertex in our setting is roughly of size $n^2p = \Theta(\sqrt
n\log n)$, which is much below the point at which we can rely on regularity
properties. Hence, we first need to show using ad-hoc density techniques that
$x$ and $y$ expand to $\Theta(n)$ vertices in two hops, and from then on start
implementing the strategy outlined above. This all also affects the design of
our $xy$-absorbers.

As a final preparation step, we need a statement for dense hypergraphs that
enables us to find a copy of the contracted backbone of an absorber in the
reduced graph.

\begin{lemma}[Proposition~8 in~\cite{buss2013minimum}]
  \label{lem:simple-absorber-dense-graph}
  For every $\gamma \in (0, 3/8)$ the following holds for every sufficiently
  large $n$. Suppose $\cH$ is a $3$-uniform hypergraph on $n$ vertices which
  satisfies $\delta_1(\cH) \geq (5/8+\gamma)^2\binom{n}{2}$. Then for every pair
  of vertices $x,y \in V(\cH)$ the number of $7$-tuples that form a copy of
  $A_{xy}^2$ is at least $(\gamma n)^7/8$.
\end{lemma}

\begin{proof}[Proof of Lemma \ref{lem:single-absorber}, case $d=1$]
  Recall, $\delta_1 = 7/16$. Given $\gamma$, let $t_0 \in \N$ be a large
  constant, in particular such that $3/t_0 \ll \gamma$. Next, let $\xi =
  \xi_{\ref{lem:large-second-neighbourhood}}(\gamma/256)$, $\lambda =
  \lambda_{\ref{lem:robust-degree-inh}}(\gamma)$, choose $\eps$ sufficiently
  small, and let $\eta$ and $b$ be as given by
  Lemma~\ref{lem:reg-lem-partition-min-degree} for $7/16+\gamma/8$ (as $\delta$)
  and other respective parameters. Pick $\mu$ sufficiently small with respect to
  $\eps$, $t_0$, and $\xi$. Let $G' := G - S$. We condition on $\Hnp$ satisfying
  the conclusion of Lemma~\ref{lem:large-second-neighbourhood},
  Lemma~\ref{lem:random-set-degree-inh}, Lemma~\ref{lem:robust-degree-inh}, and
  Lemma~\ref{lem:connecting-lemma}. Furthermore,
  \begin{alphenum}
    \stepcounter{propcnt}
    \item\label{hnp-prop-KLR} for all disjoint $U_1, U_2 \subseteq V(G)$ and
      $\cF \subseteq V(G)^4$, with $|U_i|, |\cF| \geq \xi n$, the conclusion of
      Theorem~\ref{thm:KLR} holds for $G(\cU,\cF)$ with $\cU = (U_1,U_2)$: for
      every $H$ with $m_3(H) \leq 2/3$, every subgraph $G_0$ of $G(\cU, \cF)$
      which belongs to $\cG(H, \mu n, 2 \eps p \mu^3 n^3, p, \eps)$ contains a
      canonical copy of $H$;
    \item\label{hnp-prop-upper-uniform} every $G(\cU,\cF)$ as in
      \ref{hnp-prop-KLR} is $(\eta, b, p)$-upper-uniform;
  \end{alphenum}

  Let us establish \ref{hnp-prop-KLR} and \ref{hnp-prop-upper-uniform}.
  Observe that $G(\cU, \cF)$ has $\Omega(n)$
  vertices and that it can be coupled with a random graph on its
  vertex set with edge
  probability $p$. Namely, there is a bijection $\phi$ from edges of $G(\cU,
  \cF)$ to $3$-sets in $V(G)$ such that the existence of $e$ as an edge in
  $G(\cU, \cF)$ is determined by $\phi(e)$ being an edge in $G$ or not.

  So, for a fixed choice of $U_1$, $U_2$, and $\cF$, the graph $G(\cU,\cF)$
  satisfies both \ref{hnp-prop-KLR} and \ref{hnp-prop-upper-uniform} with
  probability at least $1 - e^{-\Omega(n^3p)} - e^{-\omega(n\log n)}$. As there
  are at most $2^{2n} \cdot n^{4 n} \leq e^{10n \log n}$ choices for these
  sets, recalling that $p \geq Cn^{-3/2}\log n$, a simple union bound shows that
  w.h.p.\ both \ref{hnp-prop-KLR} and \ref{hnp-prop-upper-uniform} hold in
  $\Hnp$. Thus, from now on we also condition on
  \ref{hnp-prop-KLR}--\ref{hnp-prop-upper-uniform}.

  By the property of Lemma~\ref{lem:robust-degree-inh} there is a set $T
  \subseteq V(G) \setminus \{x, y\}$ with $|T| \leq \sqrt{n}$ such that the
  graph $G'' := G - (S \cup T)$ has minimum degree at least
  $(7/16+\gamma/4)p\binom{|V(G'')|-1}{2}$. Hence, for simplicity of notation, we
  assume that $G'$ already satisfies this.

  Let $V_x \cup V_y \cup U_1 \cup U_2 \cup W$ be an equipartition of $V(G')
  \setminus \{x,y\}$ such that $\deg_{G'}(v, Z) \geq
  (7/16+\gamma/8)p\binom{|Z|-1}{2}$ for every $v \in V(G')$ and $Z \in \{V_x,
  V_y, U_1, U_2, W\}$. Observe that $|Z| \geq n/6$. A vast number of partitions
  is such by the property of Lemma~\ref{lem:random-set-degree-inh}, so we fix
  one of them. In order to complete the proof it is sufficient to show that
  \begin{enumerate}[label=(\textit{\roman*}), ref=(\textit{\roman*})]
    \item\label{single-absorb-1} there exists a backbone of $A^1_{xy}$ in
      $G'[V_x \cup V_y \cup U_1 \cup U_2 \cup \{x,y\}]$;
    \item\label{single-absorb-2} for a copy of a backbone of $A^1_{xy}$ as
      above, there is an $x_1x_{10}$-path, a $y_1y_{10}$-path, a $x_5y_5$-path,
      and a $v_7b_1$-path, each of length four, and whose internal vertices all
      belong to $W$.
  \end{enumerate}
  Throughout the proof and for ease of reference, it might help to have
  Figure~\ref{fig:absorber-1} and especially
  Figure~\ref{fig:absorber-without-paths} in mind. Assuming we have found the
  backbone of $A_{xy}^1$, step \ref{single-absorb-2} follows by applying the
  Connecting Lemma with $W$ and (the images of) $\{\{x_1, x_{10}\}, \{y_1,
  y_{10}\}, \{x_5, y_5\}, \{v_7, b_1\}\}$ as the family of pairs to be connected
  by paths. For the rest of the proof, we focus on showing
  \ref{single-absorb-1}.

  Note that, for $* \in \{x,y\}$,
  \[
    \deg_{G'}(*, V_*) \geq (7/16+\gamma/8)p\binom{n/6-1}{2} \geq
    \frac{\gamma}{256} p\binom{n-1}{2}.
  \]
  By the property of Lemma~\ref{lem:large-second-neighbourhood} for $V_*$ (as
  $W$), for every $* \in \{x, y\}$ there exist a family of $4$-tuples $\cF_*
  \subseteq V_*^4$ and a set of pairs $\cP_* \subseteq \binom{V_*}{2}$, all
  pairwise disjoint, of size $|\cF_*| = \xi n$, $|\cP_*| = \sqrt{n}$, and such
  that
  \begin{itemize}
    \item $*uv\in E(G')$, for every $\{u, v\} \in \cP_*$, and
   \item for every $(w_1,w_2,w_3,w_4) \in \cF_*$ there is some $\{u, v\} \in
     \cP_*$ with $uw_1w_2, v w_3 w_4 \in E(G')$.
  \end{itemize}

  Let now $J := G'(\{U_1, U_2\}, \cF_x \cup \cF_y)$. Recall, $J$ is then
  obtained from $G'$ by `contracting' each $4$-tuple in $\mathcal{F}_x \cup
  \mathcal{F}_y$ into a single vertex and keeping only specifically selected
  edges. Note that $\cF_*$ become sets of vertices in the contracted graph $J$.
  To reiterate, $J$ contains all edges in $G'[U_1 \cup U_2]$ and for every $\mbw
  = (w_1, w_2, w_3, w_4) \in \cF_x \cup \cF_y$ and $u, v \in U_i$, an edge $\mbw
  uv$ exists in $J$ if and only if $i = 1$ and $w_2uv \in E(G')$ or $i = 2$ and
  $w_4uv \in E(G')$. As every vertex $v \in V(G')$ satisfies $\deg_{G'}(v, U_i)
  \geq (7/16+\gamma/8)p\binom{|U_i|-1}{2}$, it follows that every $v \in V(J)$
  has its degree into $U_i$ determined by the same quantity.

  If we were to find a contracted backbone of an absorber (see
  Figure~\ref{fig:contracted-backbone}) in $J$ with $\mbw_x \in \cF_x$ (as image
  of $x'$) and $\mbw_y \in \cF_y$ (as image of $y'$), we would be done. Indeed,
  let $\mbw_x = (x_1, x_2, x_5, x_6)$. By the choice of $\cF_x$ and $\cP_x$,
  there is a pair $\{x_3, x_4\} \in \cP_x$ for which $x_1x_2x_3, x_4x_5x_6,
  xx_3x_4 \in E(G')$. Analogously, there are $y_1, \dotsc, y_6$ for which
  $y_1y_2y_3, y_4y_5y_6, yy_3y_4 \in E(G')$ and $\mbw_y = (y_1, y_2, y_5, y_6)$.
  This, using the fact that edges of (the backbone in) $J$ are actually also
  edges in $G'$, constructs a copy of the backbone of $A_{xy}^1$ in $G'$ (once
  again, see Figure~\ref{fig:absorber-backbone}).

  With this in mind, it remains to show that $J$ contains a copy of the
  contracted backbone of an absorber for some $\mbw_x \in \cF_x$ and $\mbw_y \in
  \cF_y$ as $x'$ and $y'$, respectively. Recall, by
  \ref{hnp-prop-upper-uniform}, $J$ is $(\eta, b, p)$-upper-uniform. Thus, we
  can apply Lemma~\ref{lem:reg-lem-partition-min-degree} (the sparse regularity
  lemma) to it with $h = 4$, $\xi$ (as $\lambda$), $7/16+\gamma/8$ (as
  $\delta$), and $\cF_x, \cF_y, U_1, U_2$, to get an $(\eps, p)$-regular
  partition $(V_i)_{i \in [t]}$ of $V(J)$ and a corresponding reduced graph $\cR
  := \cR((V_i)_{i \in [t]}, \eps, p, 2\eps)$.

  Let $\cU_1, \cU_2, \cZ_1, \cZ_2$ be the sets of vertices in $\cR$ whose
  corresponding clusters fully lie in $U_1, U_2, \cF_x, \cF_y$, respectively.
  Note then that, for small enough $\eps$, each $\cU_i, \cZ_i$ contains at least
  $2\xi t/3$ vertices (clusters). Additionally, all but at most $\sqrt\eps t$
  vertices $\cS \in \cU_i \cup \cZ_i$ satisfy $\deg_\cR(\cS, \cU_j) \geq
  (7/16+\gamma/16)\binom{|\cU_j|-1}{2}$ (here we used that $\eps$ and $3/t_0$
  are both extremely small with respect to $\gamma$). In fact, as $|\cU_i| \geq
  2\xi t/3$ and by choosing $\eps$ much smaller than $\xi$, by removing these at
  most $\sqrt{\eps} t$ vertices from each $\cU_i \cup \cZ_i$ of $\cR$, we obtain
  a subgraph $\cR'$ which satisfies the above minimum degree condition with a
  negligible reduction. All in all, we have a graph $\cR' \subseteq \cR$ with
  the following properties:
  \stepcounter{propcnt}
  \begin{alphenum}
    \item\label{cluster-graph-prop-1} there are at least $\xi t/2$ vertices in
      each $\cU_i, \cZ_i$, and
    \item\label{cluster-graph-prop-3} every $\cS \in \cU_i \cup \cZ_i$ satisfies
      $\deg_{\cR'}(\cS, \cU_j) \geq (7/16+\gamma/32)\binom{|\cU_j|}{2}$, for $i,
      j \in \{1,2\}$.
  \end{alphenum}

  We now constructively find a contracted backbone of an absorber in the cluster
  graph $\cR'$, after which we use \ref{hnp-prop-KLR} to finally complete the
  proof by finding a corresponding canonical copy of it in $J$. Let $X' \in
  \cZ_1$ and $Y' \in \cZ_2$.

  \begin{claim}
    There exists a contracted backbone of an absorber in $\cR'$ with $x'$ mapped
    to $X'$ and $y'$ to $Y'$.
  \end{claim}
  \begin{proof}
    Choose distinct arbitrary vertices $X_7, X_8, X_9, X_{10}, Y_7, Y_8, Y_9,
    Y_{10} \in \cU_2$ such that
    \[
      \{X', X_7, X_8\}, \{X_8, X_9, X_{10}\}, \{Y', Y_7, Y_8\}, \{Y_8, Y_9,
      Y_{10}\} \in E(\cR')
    \]
    which must exist due to \ref{cluster-graph-prop-3}.

    We can now apply Lemma~\ref{lem:simple-absorber-dense-graph} with $\cR_2 :=
    \cR'[\cU_2 \setminus \{ X_8, X_9, X_{10}, Y_8, Y_9, Y_{10}\}]$ (as $\cH$)
    and $X_7, Y_7$ (as $x, y$), since
    \[
      \delta_1(\cR_2) \geq
      \Big(\frac{7}{16}+\frac{\gamma}{32}\Big)\binom{|\cU_2|}{2} - 6|\cU_2| - 36 \geq
      \Big(\frac{5}{8}+\gamma'\Big)^2\binom{|V(\cR_2)|}{2}
    \]
    for some small $\gamma' > 0$. (Note that here we used
    \ref{cluster-graph-prop-1} and that $t$ can be assumed to be sufficiently
    large.) We conclude that there is a copy of $A^2_{X_7 Y_7}$ in $\cR_2$.

    Next, we apply Lemma~\ref{lem:switcher-dense-graph} with $\cR[\cU_1 \cup \{
    X', X_9 \}]$ (as $\cH$) and $X', X_9$ (as $u, v$) to find vertices $A_1,
    A_2, A_3, A_4 \in \cU_1$ so that
    \[
      \{A_1, A_2, A_3\}, \{A_2, A_4, X'\}, \{A_3, A_4, X_9\} \in E(\cR').
    \]
    Almost analogously, there are $B_1, B_2, B_3, B_4 \in \cU_1$ so that $\{B_1,
    B_2, B_3\}, \{B_2, B_4, Y'\}, \{B_3, B_4, Y_9\} \in E(\cR')$. This forms the
    contracted backbone of an absorber in $\cR' \subseteq \cR$ (see
    Figure~\ref{fig:contracted-backbone}).
\end{proof}

  As the $3$-density of the contracted backbone of an absorber is at most $2/3$
  by Claim~\ref{cl:m3-density}, we can use \ref{hnp-prop-KLR} in the subgraph of
  $J$ induced by the clusters corresponding to the found contracted backbone of
  an absorber in $\cR'$. This gives us a canonical copy of the contracted
  backbone of an absorber in $J$ with vertices $x'$ and $y'$ mapped into $X'$
  and $Y'$ as desired, and this finally completes the proof.
\end{proof}

\section{Putting everything together: Proof of Theorem~\ref{thm:main-theorem}}
\label{sec:main-proof}

With all the preparatory lemmas at hand, the proof of our main theorem follows
the usual steps:
\begin{enumerate*}[(i)]
  \item find an appropriate set $R$ and a not-too-large $(a,b,R)$-absorber in $G$;
  \item cover almost everything else by $o(n)$ loose paths;
  \item use the Connecting Lemma to patch those paths together over the set $R$;
  \item absorb the unused vertices of $R$ into a loose Hamilton cycle.
\end{enumerate*}

\begin{proof}[Proof of Theorem~\ref{thm:main-theorem}]
  Choose $\lambda > 0$ sufficiently small with respect to $\gamma$ so that the
  argument below works out; in particular, $\lambda <
  \min\{\lambda_{\ref{lem:robust-degree-inh}}(\gamma),
  \lambda_{\ref{lem:robust-codegree-inh}}(\gamma)\}$. Next, $\alpha =
  \alpha_{\ref{lem:absorbing-lemma}}(\gamma,\lambda)$, $\eps =
  \eps_{\ref{lem:connecting-lemma}}(\gamma/2, \alpha)$, and $\rho$ small enough
  with respect to $\eps$ and $\alpha$. Condition on $\Hnp$ having the properties
  of Lemma~\ref{lem:random-set-degree-inh}, Lemma~\ref{lem:robust-degree-inh},
  Lemma~\ref{lem:robust-codegree-inh}. Lemma~\ref{lem:almost-path-cover},
  Lemma~\ref{lem:connecting-lemma}, and Lemma~\ref{lem:absorbing-lemma}.

  Pick a random set of vertices $R \subseteq V(G)$ of size $|R| = \alpha n$. By
  the Absorbing Lemma (Lemma~\ref{lem:absorbing-lemma}) we get, w.h.p.\ over the
  choice of $R$, an $(a,b,R)$-absorber $A$, for some $a, b \in V(G) \setminus
  R$, of size at most $\lambda n$. Furthermore, by
  Lemma~\ref{lem:random-set-degree-inh}, w.h.p.\ we have that every $d$-set $S
  \subseteq V(G)$ satisfies $\deg_G(S, R) \geq (\delta_d +
  \gamma/2)p\binom{|R|-d}{3-d}$. In particular, as $\Hnp$ has the property of
  the Connecting Lemma (Lemma~\ref{lem:connecting-lemma}), the set $R$ can be
  used as the `reservoir' (set $W$ in the lemma) to find paths through it. We
  fix such a choice of $R$ and $A$ for the remainder of the proof.

  For step~(ii) of the strategy, we aim to cover almost all of $G - V(A)$ with a
  few loose paths, using Lemma~\ref{lem:almost-path-cover}. If $d = 1$ then by
  Lemma~\ref{lem:robust-degree-inh} for $V(A)$ (as $S$) there exists a set $T$
  of size $\sqrt{n}$ such that $F := G - (V(A) \cup T)$ satisfies $\delta_1(F)
  \geq (7/16+\gamma/4)p\binom{|V(F)|-1}{2}$. Otherwise, if $d=2$, set
  $T:=\varnothing$, and by Lemma~\ref{lem:robust-codegree-inh} again for $V(A)$
  (as $S$), in $F := G - V(A)$, all but at most $O(\log n/p) = o(n^2)$ pairs of
  vertices $u,v \in V(F)$ have $\deg_F(u,v) \geq (1/4 + \gamma/2)p(n-|V(A)|-2)$.
  Thus, in both cases the graph $F$ satisfies the requirements of
  Lemma~\ref{lem:almost-path-cover} and, therefore, there are $k \leq \rho n$
  disjoint loose paths that cover its vertices. Denote the $i$-th such path by
  $P_i$ and let $x_i, y_i$ be its endpoints (note, for some $i$ we may have $x_i
  = y_i$).

  Step~(iii) is to use the Connecting Lemma (Lemma~\ref{lem:connecting-lemma})
  to connect everything into a large loose cycle. Let $T = \{v_1, \dotsc, v_t\}$
  (note, $t = 0$ if $d = 2$). We want to apply it with pairs:
  \[
    \{y_i,x_{i+1}\}_{i \in [k-1]} \cup \{y_k, v_1\} \cup \{v_i, v_{i+1}\}_{i \in
    [t-1]} \cup \{v_t, a\} \cup \{b, x_1\}.
  \]
  Recall, set $R$ is chosen so that $|R| = \alpha n$ and each $d$-set $Q
  \subseteq V(G)$ has $\deg_G(Q, R) \geq
  (\delta_d+\gamma/2)p\binom{|R|-d}{3-d}$. Furthermore, the total number of
  pairs to connect is at most $k + t + 1 \leq \sqrt{n} + \rho n + 1 \leq
  \eps|R|$. Therefore, by Lemma~\ref{lem:connecting-lemma} all these pairs can
  be connected via disjoint loose paths, each of length at most $4$, whose
  internal vertices belong to $R$. These paths use a subset $R' \subseteq R$
  with $|R'| \leq 100\rho n < \alpha n/2 = |R|/2$. The union of these paths with
  $P_1, \dotsc, P_k$ makes up a loose $ab$-path $P^*$ with vertex set $(V(G)
  \setminus V(A)) \cup \{a,b\} \cup R'$. Notice that this implies $n - |V(A)| +
  |R'|$ is odd, and since $n$ is even, $|V(A)| - |R'|$ must be odd as well.
  Lastly, we use the absorbing property of $A$ to find a loose $ab$-path $P_A$
  with $V(P_A) = V(A) \setminus R'$. The union of $P^*$ and $P_A$ gives us a
  loose Hamilton cycle in $G$ as desired.
\end{proof}

{\small \bibliographystyle{abbrv} \bibliography{Transference_for_loose_Hamilton_cycles_in_random_3-uniform_hypergraphs}}

\appendix
\section{Complementary proofs}

\begin{proof}[Proof of Lemma~\ref{lem:template}]
  The proof is almost identical to the proof of
  \cite[Lemma~7.3]{ferber2022dirac} with $k=2$. We also make use of the
  following result from \cite{montgomery2019spanning}.

  \begin{lemma}[{\cite[Lemma~10.7]{montgomery2019spanning}}]\label{bipartiteResilientMatching}
    For any sufficiently large $s$, there exists a bipartite graph $R$ with
    vertex parts $X$ and $Y\cup Z$, where $Y$ and $Z$ are disjoint, with
    $|X|=3s$ and $|Y|=|Z|=2s$, and maximum degree 100, such that if we remove
    any $s$ vertices from $Z$, the resulting bipartite graph has a perfect
    matching.
  \end{lemma}

  We construct $T$ by starting with the bipartite graph $R$ from
  Lemma~\ref{bipartiteResilientMatching} with $s=\ceil{m/2}$ and potentially
  deleting one vertex from $Z$ to ensure it has size $m$. Let $G$ be an
  $m$-vertex graph with maximum degree $4$ and no independent set of size $m/2$
  (e.g.\ take the square of a cycle on $m$ vertices) and add to $T$ the edges of
  $G$, placed on the vertex set $Z$. Thus, $T$ has at most $7s \leq 4m$
  vertices, and its maximum degree is at most $100 + 4 = 104$.

  Consider a set $Z' \subseteq Z$ with $|Z'| < m/2$ such that $V(T) \setminus
  Z'$ has even cardinality. Since $G$ has no independent set of size $s$, we can
  construct a matching $M_1$ in $T[Z\setminus Z']$ by repeatedly taking away
  edges one by one until precisely $s$ vertices are left. These remaining $s$
  vertices of $Z$ along with $X$ and $Y$ induce a subgraph in $T$ which contains
  a perfect matching $M_2$ by Lemma~\ref{bipartiteResilientMatching}. Therefore,
  $M_1 \cup M_2$ is a perfect matching of $T - Z'$.
\end{proof}

\begin{proof}[Proof of Claim~\ref{cl:m3-density}]
  We first prove a claim that allows us to split the contracted
  backbone of an absorber into three subgraphs and consider each of them separately.

  \begin{claim}\label{splitGraphM3Density}
    Let $H$ be a $3$-graph consisting of two linear $3$-graphs $H_1$ and $H_2$
    intersecting in a single vertex. Then if $\alpha \geq 1/2$ and $m_3(H_i)
    \leq \alpha$ for $i \in \{1,2\}$, we have $m_3(H) \leq \alpha$.
  \end{claim}
  \begin{proof}
    Let $H' \subseteq H$ with $v(H') > 3$ and let $H_i' = H' \cap H_i$. Note
    that $e(H') = e(H'_1)+e(H'_2)$ and $v(H') \geq v(H'_1) + v(H'_2) - 1$. It
    suffices to consider only cases where $H'$ intersects both $H_1$ and $H_2$
    in at least one edge (and thus each $H'_i$ has at least three vertices).
    If this is not the
    case for some $i\in \{1, 2\}$, we have $m_3(H') \leq m_3(H'_{3-i})\leq
    \alpha$. If $v(H'_i) > 3$, by $m_3(H_i) \leq \alpha$ it follows that
    $e(H'_i) \leq \alpha v(H'_i) - 3 \alpha + 1$. If $v(H'_i)=3$, then $e(H'_i)
    = 1$, and $e(H'_i) \leq \alpha v(H'_i) - 3 \alpha + 1$ holds as well.
    Therefore,
    \[
      \frac{e(H') - 1}{ v(H') - 3} \leq \frac{e(H'_1)+e(H'_2)-1}{v(H'_1) +
      v(H'_2) - 4} \leq \frac{\alpha (v(H'_1) + v(H'_2)) -6\alpha+1 }{v(H'_1) +
      v(H'_2) - 4} \leq \alpha,
    \]
    where the last inequality follows from $\alpha \geq 1/2$.
  \end{proof}

  Consider the components of the contracted backbone of an absorber obtained after
  removing the edges of $A_{x_7y_7}^2$ (see
  Figure~\ref{fig:contracted-backbone}). Denote the component which contains
  $x'$, respectively $y'$, by $S_{x'}$, respectively $S_{y'}$. By
  Claim~\ref{splitGraphM3Density}, it is enough to show that the $3$-density of
  $A^2_{x_7y_7}$ and of $S_{x'}$ are each at most $2/3$.

  To do this, consider a $3$-graph $H$ that is a subgraph of either
  $A^2_{x_7y_7}$ or $S_{x'}$. It suffices to consider only subgraphs with no
  isolated vertices and with $e(H) \geq 2$. If $e(H) = 2$, we have $v(H) \geq 5$
  since $A^2_{x_7y_7}$ and $S_{x'}$ are both linear hypergraphs, so
  $\frac{e(H)-1}{v(H)-3} \leq \frac{2}{3}$. If $e(H) = 3$, then $v(H) \geq 6$
  since the union of any two edges contains at least $5$ vertices, and an extra
  edge requires an extra vertex (otherwise it must have at least two vertices in
  common with at least one of the first two edges). Thus, $\frac{e(H) - 1}{ v(H)
  -3 } \leq \frac{2}{3}$ in this case as well. If $e(H)=4$, we have $v(H) \geq
  8$ since $A^2_{x_7y_7}$ and $S_{x'}$ each have $9$ vertices and $5$ edges, and
  in each of them there is no edge with two vertices of degree $1$. Thus,
  $\frac{e(H) - 1}{ v(H) -3 } \leq \frac{3}{5} < \frac{2}{3}$ again. Finally,
  $e(H)=5$ implies $H$ is either $A^2_{x_7y_7}$ or $S_{x'}$ and so $v(H) = 9$
  and $\frac{e(H)-1}{v(H)-3} = \frac{4}{6} = \frac{2}{3}$.
\end{proof}

\end{document}